\documentclass[11pt]{amsart}
\usepackage{lmodern}
\usepackage{amsmath, amsthm, amssymb, amsfonts}
\usepackage[normalem]{ulem}
\usepackage{mathrsfs}
\usepackage{verbatim} 
\usepackage{longtable}
\usepackage{mathtools}

\usepackage{tikz}
\usetikzlibrary{decorations.pathmorphing}
\tikzset{snake it/.style={decorate, decoration=snake}}
\usepackage{caption}
\usepackage{tikz-cd}
\usetikzlibrary{arrows}
\usepackage{hyperref}
\usepackage{url}

\theoremstyle{plain}
\newtheorem{thm}{Theorem}[section]
\newtheorem{cor}[thm]{Corollary}
\newtheorem{lem}[thm]{Lemma}
\newtheorem{prop}[thm]{Proposition}
\newtheorem{conj}[thm]{Conjecture}

\theoremstyle{definition}
\newtheorem{defn}[thm]{Definition}
\newtheorem{example}[thm]{Example}

\theoremstyle{remark}
\newtheorem{rmk}[thm]{Remark}

\usepackage{tikz}
\usepackage{lmodern}
\usetikzlibrary{decorations.pathmorphing}

\makeatletter
\let\@wraptoccontribs\wraptoccontribs

\begin{document}
\title{On the Perversity of Chern Classes for Compactified Jacobians}
\date{}

\author{Soumik Ghosh}
\address{Yale University}
\email{soumik.ghosh@yale.edu}
\date{}

\begin{abstract}
We prove some perversity bounds for the Chern classes of a compactified Jacobian fibration, namely that the $k$-th Chern class of the compactified Jacobian has perversity $\leq k$. Our results are motivic in nature and we also prove a filtration version of a conjecture raised in \cite{bae2024generalizedbeauvilledecompositions}. 
\end{abstract}

\maketitle

\setcounter{tocdepth}{1} 

\tableofcontents
\setcounter{section}{-1}

\section{Introduction}
Throughout, we work with varieties/schemes over the complex numbers $\mathbb C$. 
\subsection{Overview}\label{sec0.1}
If $C$ is an integral, projective curve, the Jacobian $J:=\operatorname{Pic}^0(C)$ is not necessarily proper, but it admits a natural compactification $\bar J\supset J$ as in \cite{MR429908}, \cite{MR498546}. By definition, $\bar J$ is the moduli space of torsion-free sheaves $F$ on $C$ such that $F$ has generic rank one and $\chi(F) = \chi(\mathcal O_C)$. Now let $B$ be a non-singular variety and $\mathcal C\rightarrow B$ be a flat family of integral projective curves with a section. We assume that each fiber is an integral curve with arithmetic genus $g$ and at worst planar singularities. Then, one can consider the relative compactified Jacobian fibration $\bar J_{\mathcal C}\xrightarrow{\pi}B$. We further assume $\bar J_{\mathcal C}$ is non-singular and that $\mathcal C \rightarrow B$ admits a section whose image is contained in the smooth locus of the fibers. 
Studying the interaction between natural cohomology classes and the perverse filtration has been an important topic in recent years; for example the $P=W$ conjecture in non-abelian Hodge theory has been proved by reducing to the interaction between tautological classes and the perverse filtration for the Hitchin system as in \cite{MR4792069}, \cite{hausel2025pwmathcalh2}, \cite{MR4877374}. For Lagrangian fibrations, the interaction between the Chern classes of the tangent bundle and the (motivic) perverse filtration is closely related to the Beauville-Voisin conjectures, as explained in \cite{bae2024generalizedbeauvilledecompositions}. The purpose of this paper is to study systematically the interaction between Chern classes of the tangent bundle and perverse filtrations for compactified Jacobian filtrations, both cohomologically and motivically.
In particular, a natural question is to give bounds for the perversity of the Chern classes $c_k(T_{\bar J_{\mathcal C}})$ both at cohomological and motivic level. 

In {Section \ref{sec1.1}}, we review the notion of a motivic perverse filtration as in \cite{MR4877374} and define a natural family of motivic perverse filtrations $\left \{P^{(\alpha)}_k\mathfrak h(\bar J_{\mathcal C}) \right \}_k$ indexed by $\alpha \in \mathbb Q$. The $\alpha$ in the motivic perverse filtration occurs  because of the flexibility of the Todd convention that arises due to Grothendieck-Riemann-Roch; indeed, for instance, in \cite{MR4877374}, the authors worked with the case $\alpha=0$ while in \cite{bae2024generalizedbeauvilledecompositions}, the authors worked with $\alpha=\frac{1}{2}$. The multiplicativity of the motivic perverse filtration has been shown for the case $\alpha=0$ in \cite{MR4877374} as a consequence of Fourier Vanishing (see \ref{def1.2}). In this paper, we show that each of the motivic perverse filtrations $\left \{ P_k^{(\alpha)}\mathfrak h(\bar J_{\mathcal C})\right \}_k$ is multiplicative as a consequence of a generalized Fourier Vanishing (see Equation \ref{eq37}). Furthermore, each motivic perverse filtration $\left \{P^{(\alpha)}_k \mathfrak h(\bar J_{\mathcal C})\right \}_k$ specializes to a perverse filtration $\left \{P^{(\alpha)}_k A_\bullet(\bar J_\mathcal C)\right\}_k$on the Chow groups with $\mathbb Q$-coefficients and the usual perverse filtration on the singular cohomology $H^\bullet (\bar J_{\mathcal C}, \mathbb Q )$.  In Section \ref{sec3.2}, we show that the apriori different perverse filtrations for different $\alpha$'s are the same on Chow groups (see Proposition \ref{prop3.2}), so the Chow-theoretic perverse filtration is intrinsic in some sense, which we denote by $\left\{P_kA^\bullet (\bar J)\right\}_k$. Regarding the perversity of the Chern classes $c_k (T_{\bar J_{\mathcal C}})$, one has the trivial bound $c_k(T_{\bar J_{\mathcal C}})\in P_{2k}H^k(\bar J_{\mathcal C}, \mathbb Q)$. For a general proper morphism between varieties over $\mathbb C$ say $X\xrightarrow {f}B$, given $\mathfrak z\in H^k(X,\mathbb Q)$, in general it is challenging to prove optimal non-trivial perversity bounds for $\mathfrak z$ following the usual definitions and methods as described in \cite{MR3752459} or \cite{MR2525735}. In this paper, we give optimal bounds for the Chern classes for a compactified Jacobian fibration using the interactions between Fourier-Mukai transforms and the perverse filtration as in \cite{MR4877374} and \cite{bae2024generalizedbeauvilledecompositions}. 

Our main result is the following:
\begin{thm}\label{thm0.1}
    Let $\bar J_{\mathcal C}\xrightarrow{\pi} B$ be as above. Then \begin{equation*} \begin{split} &c_k(T_{\bar J_{\mathcal C}})\in P_k A^k(\bar J_\mathcal C)=A^k(P^{(\alpha)}_k\mathfrak h(\bar J_{\mathcal C})),  \\ 
        & c_k(T_{\bar J_{\mathcal C}})\in P_kH^{2k}(\bar J_{\mathcal C}, \mathbb Q) 
         \end{split}
    \end{equation*} where the equality in the first statement follows from Proposition \ref{prop3.2}.
\end{thm}
Since the perverse filtrations $\left\{P_k A^\bullet\right\}_k $ and $\left\{P_k H^\bullet\right\}_k $ are multiplicative with respect to the cup-product (see Section \ref{sec3.3}), Theorem \ref{thm0.1} also locates any monomials of the Chern classes in the perverse filtration, namely \begin{equation*}
\begin{aligned}
  &  \prod_{i=1}^s c_{k_i}(T_{\bar J_{\mathcal C}}) \in P_{\sum_{i=1}^s k_i} A^{\sum_{i=1}^s k_i} ( \bar J_{\mathcal C} ), \\    &  \prod_{i=1}^s c_{k_i}(T_{\bar J_{\mathcal C}}) \in P_{\sum_{i=1}^s k_i} H^{\sum_{i=1}^s 2k_i} ( \bar J_{\mathcal C}, \mathbb Q ).
\end{aligned}
\end{equation*} 

We briefly explain how this relates to some conjectures in \cite{bae2024generalizedbeauvilledecompositions}. We state the conjectures in the language of relative Chow motives of Corti-Hanamura; see Section \ref{sec1.1} for the meanings of the notations.    

\begin{conj}\label{conj0.2} [\cite{bae2024generalizedbeauvilledecompositions} Conjecture $3.1$, Conjecture $3.6$] Let $\bar J_{\mathcal C} \xrightarrow{\pi}B$ be as above i.e. a compactified Jacobian fibration associated to a flat family of integral projective locally planar curves of arithmetic genus $g$ over a non-singular base $B$.  We further assume that $\bar J_{\mathcal C}$ is non-singular and $\pi$ is Lagrangian with respect to a holomorphic symplectic form on $\bar J_{\mathcal C}$. Then there exists a decomposition \begin{equation*}
    \mathfrak h(\bar J_{\mathcal C})= \bigoplus_{i=0}^{2g} \mathfrak h_i (\bar J_{\mathcal C}) \in  \operatorname{CHM}(B), \ \mathfrak h_i(\bar J_{\mathcal C}) =(\bar J_{\mathcal C}, \bar{\mathfrak p}_i,0)
\end{equation*}whose homological realization splits the perverse filtration on $R\pi_* \mathbb Q_{\bar J_{\mathcal C}}$, and satisfies: 
\begin{enumerate}
    \item stability under Fourier transform i.e. $$\bar {\mathfrak p}_j \circ \mathfrak F \circ \bar{\mathfrak p}_i = 0 , \ i+j\neq 2g$$
    \item multiplicativity with respect to the cup product i.e. $$\bar{\mathfrak p}_k \circ [\Delta^{sm}_{\bar J_{\mathcal C}/B}]\circ (\bar {\mathfrak p}_i \times  \bar{\mathfrak p}_j )=0 $$
    \item $c_{2i}(T_{\bar J_{\mathcal C}}) \in A^{2i}\mathfrak h_{2i}(\bar J_{\mathcal C})$. 
\end{enumerate}
\end{conj}
Conjecture \ref{conj0.2} implies that for such a compactified Jacobian fibration $c_k(T_{\bar J_{\mathcal C}})\in P_k H^{2k}(\bar J, \mathbb Q)$ for all $k$. We prove this perversity bound for any compactified Jacobian fibration $\bar J_{\mathcal C}$ with $\bar J_{\mathcal C}$ non-singular (without any Lagrangian assumption), on a motivic level, thereby proving the filtration version of the above conjecture. Note that an analogous perversity bound for the Hitchin system is a consequence of the $P=W$ Theorem. 

To simplify notation, we shall henceforth denote $\bar J_{\mathcal C}$ by $\bar J$ and $J_{\mathcal C}$ the Jacobian fibration by $J$. So $J$ is an open dense subset of $\bar J$ and forms a commutative group scheme over $B$.  
\subsection{Outline of the Paper} In {Section \ref{sec1.1}}, we review the relative Chow motives of Corti-Hanamura, the motivic perverse filtration and construct the family of motivic perverse filtrations, mentioned in {Section \ref{sec0.1}}. In {Section \ref{sec2}}, we recall the existence and properties of a Poincare sheaf on $\bar J \times_B \bar J$, as constructed by Arinkin in \cite{MR3019453} and recall the notion of a Fourier transform. We also prove some codimension estimates. In Sections \ref{sec3.1} and \ref{sec3.2}, we use the codimension estimates of Section \ref{sec2.2} to prove various vanishing relations, namely a twisted Fourier vanishing and construct a family of motivic perverse filtrations. We further show that at the level of Chow groups they induce the same filtration. In Section \ref{sec3.3} we show that the motivic perverse filtrations as constructed are multiplicative. In Section \ref{sec4.1}, we prove more vanishing relations involving Fourier transforms of the Chern classes and use the relations to prove the perversity bounds in Section \ref{sec4.2}. Finally, in Section \ref{Examples}, we give examples to show that the bound we get is optimal. 
\subsection{Acknowledgments} The author is grateful to Junliang Shen for his valuable guidance and helpful discussions, and Dima Arinkin for his valuable insight. The author would also like to thank David Fang, David Bai for helpful discussions and the anonymous referee(s) for their very careful reading of the manuscript. 

\section{Motivic Perverse Filtrations}
\subsection{Relative Chow Motives and Perverse Filtrations}\label{sec1.1} We work with the relative Chow motives of Corti-Hanamura, see \cite{MR1763656} and \cite{MR2330158} for more details. We fix a non-singular base variety $B$. All Chow groups $A_\bullet$ are taken with $\mathbb Q$-coefficients. 

Let $f_1:X_1\rightarrow B$ and $f:X_2\rightarrow B$ be two proper morphisms with $X_1,X_2$ non-singular. If $X_2$ is equi-dimensional, we define the group of degree $k$ relative
 correspondences between $X_1$ and $X_2$ to be \begin{equation*}
     \operatorname{Corr}^k_B(X_1,X_2):=A_{\dim X_2-k}(X_1\times_B X_2)
 \end{equation*}If $X_2$ is not equi-dimensional we decompose $X_2$ into connected components and define the groups accordingly. If $\Gamma\in \operatorname{Corr}^{k}_B(X_1,X_2), \ \Gamma'\in \operatorname{Corr}^{k'}_B (X_2,X_3)$, we define their composition by \begin{equation}\label{eq1}
     \Gamma'\circ\Gamma := {p_{13}}_*\delta_{X_2}^!(\Gamma \times \Gamma')\in \operatorname{Corr}^{k+k'}_B (X_1,X_3) 
 \end{equation}where $p_{13}$ is the projection onto the first and third factors $X_1\times_B X_2\times_B X_3\xrightarrow{p_{13}}X_{1}\times_B X_3$ and $$\delta^!_{X_2}:A_{\bullet}(X_1\times_B X_2\times X_2\times_B X_3)\rightarrow A_{\bullet- \dim X_2}(X_1\times_B X_2\times_B X_3)$$ is the refined Gysin homomorphism with respect to the regular embedding $\Delta_{X_2}: X_2\rightarrow X_2\times X_2$ as defined in \cite{MR1644323}.

 We also define a Chow-theoretic convolution. Say $\Gamma_i\in \operatorname{Corr}^{k_i}_B(Y_i,X_i):= A_{\dim X_i-k_i}(Y_i\times_B X_i), \ i=1,2$ and $\Gamma\in \operatorname{Corr}^l_B(X_1,X_2;X_3):=A_{\dim X_3 - l}(X_1\times_B X_2\times_B X_3)$. We want to make sense of $$\Gamma\circ (\Gamma _1 \times \Gamma_2)\in \operatorname{Corr}_B^{k_1+k_2+l}(Y_1,Y_2;X_3).$$We define this as follows: \begin{equation*}
     \Gamma\circ (\Gamma_1\times \Gamma_2) = {\pi_{Y_1\times_B Y_2\times_B X_3}}_* \delta_{X_2}^! \delta_{X_1}^!\left( \pi_{Y_1\times_B X_1}^* \Gamma_1 \pi_{Y_2\times_B X_2}^* \Gamma_2\pi_{X_1\times_B X_2\times_B X_3}^* \Gamma\right )
 \end{equation*} where $\pi_{(-)}$'s are the natural projections from $Y_1\times_B X_1 \times Y_2\times_B X_2 \times X_1\times_B X_2\times_B X_3$. So we get
 $$\Gamma\circ (\Gamma_1\times \Gamma_2)\in \operatorname{Corr}^{l+k_1+k_2}_B( Y_1,Y_2;X_3).$$We have kept our notation consistent with \cite{MR4877374} Section $2.2$. 

If $\mathfrak Z \in \operatorname{Corr}_B^\bullet(X,Y)$, then we have a morphism given by the Fourier-Mukai transform \begin{equation*}
    \begin{aligned}
       & \mathfrak Z: A_\bullet (X) \rightarrow A_\bullet (Y) \\ & z\mapsto {p_Y}_*\left (p_X^* z \cdot \iota_*\mathfrak Z\right )
    \end{aligned}
\end{equation*}where $\iota: X\times_B Y\hookrightarrow X\times Y$ is the inclusion, $p_X, p_Y$ are the two projections $X\times Y\rightarrow X, Y$. 

Under the  identification $\operatorname{Corr}_B ^\bullet(B, X)=A_\bullet(X)$, we have for $\mathfrak Z\in \operatorname{Corr}^\bullet_B(B, X)$ and $z\in \operatorname{Corr}^\bullet_B(B,B)=A_\bullet (B)$, the equality \begin{equation*}
    \mathfrak Z(z)=\mathfrak Z\circ z
\end{equation*}
where the LHS is given by the Chow-theoretic Fourier-Mukai transform and the RHS is given by the convolution of correspondences as in Equation (\ref{eq1}).

 The category of relative Chow motives, denoted by $\operatorname{CHM}(B)$ consists of triples $(X,\mathfrak p, m)$ where $X$ is a non-singular variety proper over $B$, $\mathfrak p\in \operatorname{Corr}^0_B(X,X)$ is a projector, i.e. $\mathfrak p \circ \mathfrak p =\mathfrak p$ and $m\in \mathbb Z$ takes care of Tate twists. Given any non-singular variety $X$, proper over $B$, we have the motive $\mathfrak h(X):=(X, [\Delta_{X/B}],0)$. Morphisms in this category are defined by \begin{equation*}
     \operatorname{Hom}_{\operatorname{CHM}(B)}\left((X_1,\mathfrak p,m),(X_2,\mathfrak q,n) \right):= \mathfrak q\circ \operatorname{Corr}^{n-m}_B (X_1,X_2)\circ \mathfrak p.
 \end{equation*}Given a Chow motive $\mathfrak M=(X,\mathfrak p,m)$ in $\operatorname{CHM}(B)$, we define its Chow groups by $$A^k(\mathfrak M):= \operatorname{Hom}_{\operatorname{CHM}(B)}\left (\mathfrak h(B)(-k), \mathfrak M \right )$$ where $\mathfrak h(B)(-k)=(B\xrightarrow{id}B, \Delta_{B/B}=[B], -k)\in \operatorname{CHM}(B)$. 

A key fact proven in \cite{MR1763656} is that we have a canonical isomorphism $$\varphi : \operatorname{Hom}_{D_c^b(B)}({f_1}_*\mathbb Q_{X_1}[i], {f_2}_*\mathbb Q_{X_2}[j])\xrightarrow{\simeq} H^{BM}_{2\dim X_2+i-j}(X_1\times_B X_2).$$
 
We have a well-defined Corti-Hanamura realization functor \begin{equation}
    \operatorname{CHM}(B)\rightarrow D^b_c(B),
\end{equation}sending $(X\xrightarrow{f}B,\mathfrak p, m)$ to $\mathfrak p_* \left(f_* \mathbb Q_X[2m] \right)= \operatorname{Im} (\varphi^{-1}cl( \mathfrak p): f_* \mathbb Q_X\rightarrow f_*\mathbb Q_X)[2m]$ where $cl: \operatorname{Corr}^\bullet_B(X,X)\rightarrow H^{BM}_\bullet(X\times_B X)$ is the cycle class map and $\varphi$ is the isomorphism above. This further specializes by taking global cohomology $H^\bullet(-, \mathbb Q)$. 

We briefly recall the definition of a motivic perverse filtration as introduced in \cite{maulik2024algebraiccycleshitchinsystems}.
If $X\rightarrow B$ is a proper morphism with $X,B$ non-singular varieties, we set $R_f:=\dim X\times_B X-\dim X$ to be the defect of semi-smallness.  
\begin{defn}\label{def1.1}
    We say a sequence of motives $P_k\mathfrak h(X)=(X,\mathfrak p_k,0), \ 0\leq k\leq 2R_f$, which are summands of $h(X)$ form a motivic perverse filtration for $f$ if \begin{itemize}
        \item (Termination) $P_{2R_f}\mathfrak h(X)=\mathfrak h(X)$
        \item (Realization) Under the Corti-Hanamura functor (\cite{MR1763656}), the natural inclusion \begin{equation*}P_\bullet \mathfrak h(X)\hookrightarrow \mathfrak h(X)\end{equation*} specializes to the natural inclusion induced by the perverse truncation functor \begin{equation*}
            {}^p\tau_{\leq \bullet +\dim X-R_f}Rf_* \mathbb Q_X \hookrightarrow Rf_*\mathbb Q_X\end{equation*}
        \item  (Semi-Orthogonality) $\mathfrak p_{k+1}\circ \mathfrak p_k=\mathfrak p_k$
    \end{itemize} 
    Further, we say a motivic perverse filtration is multiplicative if \begin{equation*}
        \left ([\Delta_{X/B}] - \mathfrak p_{k+l} \right ) \circ [\Delta^{sm}_{X/B}] \circ (\mathfrak p_k \times \mathfrak p_l) =0 \ \forall \ k,l
    \end{equation*}where $\Delta^{sm}_{X/B}$ is the small diagonal. 
\end{defn}
We briefly recall some facts about perverse filtrations which shall be used in the later sections.

Let $f: X \to B$ be a proper morphism between nonsingular varieties with equi-dimensional fibers. By the decomposition theorem \cite{MR4870047}, we have
\[
f_* \mathbb Q_X \simeq \bigoplus_{i} {^{p}}\mathcal H^i(f_* \mathbb Q_X)[-i] \in D^b_c(B)
\]
where $D^b_c(-)$ stands for the bounded derived category of constructible sheaves and all functors are derived. Each perverse cohomology on the right-hand side is a semisimple perverse sheaf. We say that $f$ has \emph{full support}, if each simple summand of ${^{p}}\mathcal H^i(f_* \mathbb Q_X)$ has support $B$.
\begin{prop}[\cite{MR2653248} Théorème $7.2.1$, \cite{MR3259038} Theorem $2.4$]\label{prop1.2}
    If $\bar J\xrightarrow[]{\pi}B$ is a compactified Jacobian fibration with $\bar J$ non-singular, then $\pi$ has full support.
\end{prop}
\begin{rmk}
    We note that the proof of the above proposition relies on a $\delta$-inequality for the base $B$, [see condition (A3) in Section $2.1$ of \cite{MR3259038}] which in our case is guaranteed by the Severi inequality, since we assumed $\bar J$ is non-singular; see Lemma \ref{Severi Inequality}.
\end{rmk}
In the next section, we study motivic perverse filtrations for a compactified Jacobian fibration. In particular, we construct an explicit family of motivic perverse filtrations.
\subsection{Motivic Perverse Filtrations for Compactified Jacobians}\label{sec1.2} We keep the notations as in Section \ref{sec0.1}.
\begin{defn}\label{def1.2}
    We say a pair of correspondences $\mathfrak Z,\mathfrak Z^{-1} \in \operatorname{Corr}_B^{\bullet}(\bar J,\bar J)$ is a good pair if they satisfy the following conditions: \begin{enumerate}
        \item $\mathfrak Z\circ \mathfrak Z^{-1}=[\Delta_{\bar J/B}]\in \operatorname{Corr}_B^0(\bar J,\bar J)$ and $\mathfrak Z^{-1}\circ \mathfrak Z=[\Delta_{\bar J/B}]\in \operatorname{Corr}_B^0(\bar J,\bar J)$
        \item $\mathfrak Z, \mathfrak Z^{-1}$ satisfy Fourier Vanishing  (FV) in the sense of \cite{MR4877374} i.e. \begin{equation*}\label{(FV)}\tag{FV}
            \mathfrak Z^{-1}_k\circ \mathfrak Z_l =0\text{ if }k+l<2g
        \end{equation*}where $\mathfrak Z_k\in A^k(\bar J\times_B \bar J)=\operatorname{Corr}^{k-g}_B(\bar J,\bar J)$ is the degree $k-g$ component of the correspondence $\mathfrak Z$ and similarly for $\mathfrak Z^{-1}$. 
        \item There exists a Zariski open set $U\subset B$ such that each fiber over $U$ is non-singular i.e. $\bar J_U\xrightarrow{\pi_U} U$ is a commutative group scheme over $U$ and $\mathfrak Z|_{U}= \operatorname{ch}(\mathcal L)\cap [\bar J_U\times_U \bar J_U]$ where $\mathcal L$ is the normalized Poincare line bundle on $\bar J_U \times_U \bar J_U$. 
    \end{enumerate}
\end{defn}
Note that the first conditions tell us \begin{equation}\label{eq3}
    \sum_{i=0}^{2g}\mathfrak Z_i\circ \mathfrak Z_{2g-i}^{-1}=[\Delta_{\bar J/B}]
\end{equation}and \begin{equation}
      \sum_{i=0}^{2g}\mathfrak Z^{-1}_i\circ \mathfrak Z_{2g-i}=[\Delta_{\bar J/B}]
\end{equation}because of degree reasons.
If $\mathfrak Z, \mathfrak Z^{-1}$ are as in \ref{def1.2}, we can define projectors \begin{equation}
    \mathfrak p_k:=\sum_{i\leq k}\mathfrak Z_i \circ \mathfrak Z^{-1}_{2g-i} \in \operatorname{Corr}_B^0(\bar J,\bar J)
\end{equation}
\begin{equation}
    \mathfrak q_{k+1}:= \sum_{i\geq k+1}\mathfrak Z_i\circ \mathfrak Z_{2g-i}^{-1}\in \operatorname{Corr}^0_B(\bar J,\bar J)
\end{equation}To see that $\mathfrak p_k, \mathfrak q_{k+1}$'s are indeed projectors, one observes that

\begin{equation*}
    \mathfrak p_k \circ \mathfrak p_k  = \sum_{i\leq k, j\leq k}\mathfrak Z_i \circ \mathfrak Z_{2g-i}^{-1}\circ \mathfrak Z_j\circ \mathfrak Z_{2g-j}^{-1}  = \left (\sum_{i=0}^{2g} \mathfrak Z_i \circ \mathfrak Z_{2g-i}^{-1}\right )\circ \left ( \sum_{j\leq k}\mathfrak Z_j\circ \mathfrak Z_{2g-j}^{-1}\right )  = [\Delta_{\bar J/B}]\circ \mathfrak p_k = \mathfrak p_k
\end{equation*}where the second equality follows thanks to \ref{(FV)} and the third equality follows from Equation (\ref{eq3}). For $\mathfrak q_{k+1}$, one observes that $\mathfrak q_{k+1}=[\Delta_{\bar J/B}]-\mathfrak p_k$.  It follows that \begin{equation*}
    \mathfrak p_k\circ \mathfrak q_{k+1}= \mathfrak p_k \circ ([\Delta_{\bar J/B}]-\mathfrak p_k)= \mathfrak p_k -\mathfrak p_k \circ \mathfrak p_k=0
\end{equation*}and\begin{equation*}
    \mathfrak q_{k+1}\circ \mathfrak p_k= ([\Delta_{\bar J/B}]-\mathfrak p_k)\circ \mathfrak p_k= \mathfrak p_k -\mathfrak p_k \circ \mathfrak p_k=0
\end{equation*}
\begin{prop}\label{prop1.4}
A good pair of correspondences $\mathfrak Z,\mathfrak Z^{-1}$ induces a motivic perverse filtration on $\mathfrak h(\bar J)$ by setting $P_k\mathfrak h(\bar J)=(\bar J, \mathfrak p_k, 0)$.  \end{prop}   
\begin{proof} 
We set $P_k\mathfrak h(\bar J)=(\bar J,\mathfrak p_k,0 )$. Since $\mathfrak p_{2g}=[\Delta_{\bar J/B}]$, $P_{2g}\mathfrak h(\bar J)=\mathfrak h(\bar J)$ which shows termination.

For semi-orthogonality, we note that 
\begin{gather*}
\mathfrak p_{k+1}\circ \mathfrak p_k = \left (\sum_{i\leq k+1}\mathfrak Z_i\circ \mathfrak Z_{2g-i}^{-1} \right )\circ \left (\sum_{j\leq k}\mathfrak Z_j\circ \mathfrak Z_{2g-j}^{-1}  \right ) \\ = \left (\sum_{i=0}^{2g} \mathfrak Z_i \circ \mathfrak Z_{2g-i}^{-1}\right )\circ \left ( \sum_{j\leq k}\mathfrak Z_j\circ \mathfrak Z_{2g-j}^{-1}\right )  =[\Delta_{\bar J/B}]\circ \mathfrak p_k =\mathfrak p_k
\end{gather*}
where as before, the second equality follows thanks to \ref{(FV)} and the third equality follows from Equation (\ref{eq3}). 

Now we look at the realization property. Let $U\subset B$ be as in Definition \ref{def1.2}. By definition, we have $\mathfrak Z|_U=\operatorname{ch}(\mathcal L)\cap [\bar J_U\times_U \bar J_U]$ where $\mathcal L$ is the normalized Poincare line bundle on $\bar J_U \times_U \bar J_U$ for the commutative group scheme $\bar J_U\rightarrow U$. So $\bar{\mathfrak p}_k:= {\mathfrak Z_k|}_{U}\circ {\mathfrak Z_{2g-k}^{-1}|}_U$ are projectors giving rise to the motivic decomposition in \cite{MR1133323} i.e. \begin{equation}
    \mathfrak h(\bar J_U)= \bigoplus_{i=0}^{2g} \mathfrak h_i(\bar J_U)\text{ where }\mathfrak h_i(\bar J_U)=(\bar J_U, \bar{\mathfrak p}_i,0)
\end{equation}The homological realization of the above decomposition yields a canonical decomposition into shifted local systems as shown in \cite{MR1133323}, namely \begin{equation}
            {\pi_U}_*\mathbb Q_{\bar J_U}= \bigoplus_{i=0}^{2g}R^i{\pi_U}_*\mathbb Q_{\bar J_U}[-i]
        \end{equation}As such, for the realization of $P_k\mathfrak h(\bar J)|_U=(\bar J_U, {\mathfrak p_k|}_U,0)$, we get\begin{equation}
            {({\mathfrak p_k|}_U)}_*({\pi_U}_* \mathbb Q_{\bar J_U})= \bigoplus_{i=0}^k R^i{\pi_U}_*\mathbb Q_{\bar J_U}[-i]
        \end{equation}with ${({\mathfrak p_k|}_U)}_*({\pi_U}_* \mathbb Q_{\bar J_U})\rightarrow {\pi_U}_*\mathbb Q_{\bar J_U}$ given by the natural inclusion. Consequently, for the homological realization of $P_k \mathfrak h(\bar J)\rightarrow \mathfrak h(\bar J)$ i.e. ${\mathfrak p_k}_* (\pi_* \mathbb Q_{\bar J})\rightarrow \pi_* \mathbb Q_{\bar J} $, we know that after restriction to $U$, the following vanishing of the perverse cohomology \begin{equation}\label{eq10}
           {^p\mathcal H^{i+\dim B}({\mathfrak p_k}_* (\pi_* \mathbb Q_{\bar J}))|}_U=0 \text{ if }i>k
        \end{equation}and the restriction to $U$ of the induced homomorphism \begin{equation}\label{eq11}
            ^p\mathcal H^{i+\dim B}(\bullet):{^p\mathcal H^{i+\dim B}}({\mathfrak p_k}_* (\pi_* \mathbb Q_{\bar J}))\rightarrow {^p\mathcal H^{i+\dim B}(\pi_* \mathbb Q_{\bar J})} 
        \end{equation}is an isomorphism if $i\leq k$. We know that $\pi: \bar J\rightarrow B$ has full support by Proposition \ref{prop1.2}. So, it follows that Equation (\ref{eq10}) and Equation (\ref{eq11}) must stay true over the entire base $B$; otherwise the perverse sheaf ${}^p\mathcal H^{i+\dim B}(\pi_* \mathbb Q_{\bar J})$ would have a non-trivial sub-object supported away from $U$.  
        Therefore, we get \begin{equation}{\mathfrak p_k}_* (\pi_* \mathbb Q_{\bar J})\simeq {^p\tau_{\leq k+\dim B}}\pi_* \mathbb Q_{\bar J}
        \end{equation} with ${\mathfrak p_k}_* (\pi_* \mathbb Q_{\bar J})\rightarrow \pi_* \mathbb Q_{\bar J} $ given by the natural morphism 
        $  {^p\tau_{\leq k+\dim B}}\pi_*\mathbb Q_{\bar J}\rightarrow\pi_*\mathbb Q_{\bar J}$. This completes the proof of the proposition.
\end{proof}
Note that along the same lines, one can show that if we set $Q_{k+1}\mathfrak h(\bar J):=(\bar J, \mathfrak q_{k+1},0)$ the homological realization of $\mathfrak h(\bar J)\rightarrow Q_{k+1}\mathfrak h(\bar J)$ is the natural homomorphism $\pi_* \mathbb Q_{\bar J}\rightarrow {}^p\tau_{\geq k+1+\dim B}\pi_* \mathbb Q_{\bar J}$. 

So we get a decomposition of Chow motives \begin{equation}
    \mathfrak h(\bar J)= P_k\mathfrak h(\bar J) \oplus Q_{k+1}\mathfrak h(\bar J)
\end{equation}whose homological realization gives the appropriate morphisms.
\section{Arinkin's Sheaf and Codimension Estimates}\label{sec2}
We keep the notations as before. In this section, we review Arinkin's theory of the Poincare sheaf on a compactified Jacobian and prove some codimension estimates. The crucial technical results are Proposition \ref{prop2.5}, Proposition  \ref{prop2.6} and Proposition  \ref{prop2.7}.
We begin with introducing a stratification of $B$. 
\subsection{Stratification of the Base via Geometric Genus of the Fibers} For a curve $C_b$ over a closed point $b\in B$, we look at its normalization $\tilde C_b$ and $\tilde g_b$ is the genus of $\tilde C_b$, i.e. the geometric genus of $C_b$. Let $B^{(\tilde g)}$ be the locally closed subvariety that consists of points $b\in B$ such that $C_b$ has geometric genus $\tilde g$ ie $\tilde g_b =\tilde g$.

It follows from \cite{MR961600} that the function $\Phi: B\rightarrow \mathbb Z$ given by $b\mapsto \tilde g_b$ is lower semi-continuous and hence $\left \{B^{(\tilde g)}\right \}_{\tilde g}$ gives a stratification of $B$. We end this subsection with the 'Severi inequality' which shall be used repeatedly subsequently. For $b\in B$, set $\delta(b)=g-\tilde g_b$ and for $Z\subset B$, an irreducible closed subvariety, set $\delta_Z:= \delta(b)$ for a general point of $b$. 
\begin{lem}[Severi Inequality]\label{Severi Inequality}
    If $Z\subset B$ is a closed subvariety, $\operatorname{codim}_B Z\geq \delta_Z$
\end{lem}
\begin{proof}
    This is essentially \cite{MR4602420} Lemma 4.1, which uses the fact that $\bar J$ is non-singular. 
\end{proof}
\begin{cor}\label{cor2.2}
    $\dim B^{(\tilde g)}\leq \dim B-g+\tilde g$
\end{cor}
\begin{proof}
Let $Z\subset B^{(\tilde g)}$ be an irreducible component.
    If $\bar Z$ is the closure of $Z$ in $B$, then for a general point of $b\in \bar Z$, $\delta(b)=g-\tilde g$ whence the inequality follows. 
\end{proof}
\subsection{Codimension Bounds}\label{sec2.2} In \cite{MR2915476}, Arinkin constructed the normalized Poincare line bundle $\mathcal P$ over the product $J \times_B \bar J$ extending the standard normalized Poincare line bundle for Jacobians associated with non-singular curves. In \cite{MR3019453}, it was shown that the push-forward of this sheaf via the open inclusion $J\times_B \bar J\subset \bar J\times_B \bar J$ gives a Cohen-Macaulay coherent sheaf on $\bar J\times_B \bar J$ which we still denote by $\mathcal P$ for notational convenience. Furthermore it was shown that we have an auto-equivalence induced by the Fourier-Mukai transform
 
\begin{gather*}
     FM_{\mathcal P}: D^b(\bar J)\xrightarrow{\simeq} D^b(\bar J) \\ \mathcal E\mapsto {p_2}_* (p_1^*\mathcal E\otimes \mathcal P) 
\end{gather*}where $p_1,p_2: \bar J\times_B \bar J\rightarrow \bar J$ are the two projections. The inverse of this auto-equivalence is given by $FM_{\mathcal P^{-1}}$, the Fourier-Mukai transform with kernel $$\mathcal P^{-1}=\mathcal P^\vee \otimes p_2^*\omega_{\pi}[g], \text{ where }\mathcal P^\vee:=\mathscr{H}om(\mathcal P, \mathcal O_{\bar J\times_B \bar J}).$$Here $\omega_{\pi}$ is the relative canonical bundle with respect to $\pi: \bar J\rightarrow B$.
We list the following properties of $\mathcal P$. 

First of all, we have an action of $J$ on $\bar J$, given by $\mu: J\times_B \bar J\rightarrow \bar J$. More concretely, let $b\in B$ be a closed point. Given $L_b\in \operatorname{Pic}^0(C_b)=J_b$ and $F_b\in \bar J_b$, a rank $1$ torsion-free degree $0$ sheaf on $C_b$, the action of $J$ on $\bar J$ is defined by the formula $\mu(L_b,F_b)=L_b\otimes_{\mathcal O_{C_b}} F_b\in \bar J_b$. We also have the Abel-Jacobi map  \begin{gather*}
    \mathrm{AJ}: J\times_B \mathcal C\rightarrow \bar J \\ (L,c)\in J_b\times C_b \mapsto \mathscr Hom(\mathcal I_c, L(-p))\in \bar J_b
\end{gather*}
where $\mathcal I_c$ is the ideal sheaf corresponding to the closed point $c$ and $p\in C_b$ is the smooth point given by the section of $\mathcal C\rightarrow B$. 

We recall the following properties by Arinkin.
\begin{lem}\label{lem2.3}[\cite{MR3019453} Lemma $6.4$]
   Let $ \mathcal F_b$ be a universal sheaf on $C_b\times \bar J_b$. Consider the diagram \begin{equation*}
    \begin{tikzcd}
        J_b\times \bar J_b&J_b\times C_b\times\bar J_b\arrow{l}[swap]{p_{13}}\arrow{r}{p_{23}}\arrow{d}{\mathrm{AJ}\times id} &  C_b\times J_b \\ &  \bar J_b\times \bar J_b
    \end{tikzcd}
\end{equation*}where $p_{13}, p_{23}$ are the canonical projections. Then we have $(\mathrm{AJ}\times id)^* \mathcal P|_{\bar J_b\times \bar J_b}=p_{23}^* \mathcal F_b\otimes p_{13}^* (\mathcal P|_{J_b\times \bar J_b})$.
\end{lem}
\begin{lem}\label{lem2.4}[Theorem of the Square, \cite{MR3019453} Lemma $6.5$] Consider the diagram $$\begin{tikzcd}
    J_b\times \bar J_b & J_b \times \bar J_b\times \bar J_b \arrow{l}[swap]{p_{13}}\arrow{r}{p_{23}}\arrow{d}{\mu \times id_{\bar J_b}} & \bar J_b\times\bar J_b\\ & \bar J_b\times \bar J_b
\end{tikzcd}$$where $p_{13},p_{23}$ are the canonical projections. We have $(\mu\times id_{\bar J_b})^* (\mathcal P|_{\bar J_b\times \bar J_b})= p_{13}^* \mathcal P|_{J_b\times \bar J_b} \otimes p^*_{23}\mathcal P|_{\bar J_b\times \bar J_b}$ and in particular for $(L_b, F_b)\in J_b\times \bar J_b$ over $b\in B$, we have $\mathcal P_{L_b \otimes F_b}=\mathcal P_{L_b}\otimes \mathcal P_{F_b}$ where for $G\in \bar J_b $, a closed point $\mathcal P_G =\mathcal P|_{\{ G\}\times \bar J_b}$. 
\end{lem}
We shall use these lemmas to prove the following codimension bounds. First, we introduce some notations. Let $\bar J^{n+1}$ denote the $n+1$-the relative product. It carries natural projections $$\pi_{n+1}:\bar J^{n+1}\rightarrow B, \ p_i: \bar J^{n+1}\rightarrow \bar J, \ p_{ij}:\bar J^{n+1}\rightarrow \bar J^2$$where $p_i$ is the projection onto the $i$-th factor, $p_{ij}$ is the projection onto the $i$-th and the $j$-th factor, and so on. 

\begin{prop}\label{prop2.5}
    Let $\mathcal K\in D^b(\bar J)$ underlie a $J$-equivariant bounded complex on $\bar J$. Let $\mathcal L_1, \mathcal L_2$ be $J$-equivariant vector bundles on $\bar J$. Then we have the codimension bounds\begin{gather}
        \label{eq14} \operatorname{codim}_{\bar J}\operatorname{Supp}FM_{\mathcal P\otimes p_1^*\mathcal L_1\otimes p_2^*\mathcal L_2}(\mathcal K)\geq g, \\  \operatorname{codim}_{\bar J}\operatorname{Supp}FM_{\mathcal P^{\vee}\otimes p_1^*\mathcal L_1\otimes p_2^*\mathcal L_2}(\mathcal K)\geq g.
    \end{gather} 
\end{prop}
\begin{proof}
 Since $\mathcal L_1\otimes \mathcal K\in D^b(\bar J)$ also underlies a $J$-equivariant bounded complex and using the projection formula, since $FM_{(-) \otimes p_2^*\mathcal L_2 }=\mathcal L_2 \otimes FM_{(-)}$, we see that it is enough to assume $\mathcal L_1=\mathcal L_2=\mathcal O_{\bar J}$. We proceed in the following steps: 
 
\medskip
\noindent {\bf Step 1.} Let $F\in \operatorname{Supp}FM_{\mathcal P}(\mathcal K)$ (or $FM_{\mathcal P^{\vee}}(\mathcal K)$) and $b=\pi(F)$.  Then by base change we see that for some $i$,\begin{equation}
    H^i(\bar J_b, \mathcal P_F\otimes \mathcal K_b) \neq  0  \ (\text{respectively } H^i(\bar J_b, \mathcal P_F^\vee\otimes \mathcal K_b) \neq  0 ).
\end{equation}We show that this implies $\mathcal P_F|_{J_b}$ is a trivial line bundle in the next step.

\noindent {\bf Step 2.} The argument closely follows \cite{MR2915476} Prop $1$ or \cite{maulik2023fourier} Proposition $4.2$. We know $\mathcal P_F|_{J_b}$ (respectively $\mathcal P_F^\vee|_{J_b}$) is a line bundle on $J_b$. Let $T$ (respectively $T^\vee$) be the  $\mathbb G_m$-torsor corresponding to $\mathcal P_F|_{J_b}$ (respectively $\mathcal P_F^\vee|_{J_b}$). It is naturally an abelian group which fits into the short exact sequence $$1\rightarrow \mathbb G_m \rightarrow T\rightarrow J_b\rightarrow 1 \ (\text{ respectively }1\rightarrow \mathbb G_m \rightarrow T^\vee\rightarrow J_b\rightarrow 1).$$By virtue of Lemma \ref{lem2.4}, we see that the action of $J_b$ on $\bar J_b$ lifts to an action of $T$ on $\mathcal P_F$ (respectively $T^\vee$ on $\mathcal P_F^\vee$) with $\mathbb G_m$ acting via dilation. We refer to \cite{MR4015227} Proposition $6.1$ (Claim $2$ in loc. cit.) for the details of this lifting.

Since $\mathcal K$ is $J$-equivariant, we let $T$ (respectively $T^\vee$) act on $\mathcal K_b$ via $T\rightarrow J_b$. It follows that $T$ (respectively $T^\vee$) acts on $H^i(\bar J_b, \mathcal P_F\otimes K_b)$ (respectively  $H^i(\bar J_b, \mathcal P_F^\vee\otimes \mathcal K_b)$). Thus we get $V\subset H^i(\bar J_b, \mathcal P_F \otimes K_b)$ (respectively $H^i(\bar J_b, \mathcal P_F^\vee\otimes \mathcal K_b)$) a $1$-dimensional $T$ sub-module as $T$ (respectively $T^\vee$) is commutative which corresponds to a character $\chi: T\rightarrow \mathbb G_m$ (respectively $\chi: T^\vee\rightarrow \mathbb G_m$). This gives a splitting of $\mathbb G_m\rightarrow T$ (respectively $\mathbb G_m\rightarrow T^\vee$) since $\mathbb G_m$ acts via dilation. Thus $T$ (respectively $T^\vee$) is the trivial $\mathbb G_m$-torsor which shows $\mathcal P_F|_{J_b}$ is the trivial line bundle. 

\noindent{\bf Step 3.} This step shows that $F|_{C^{sm}_b}\cong \mathcal O_{C^{sm}_b}$ where $C^{sm}_b\subset C_b$ is the smooth locus of $C_b$.

Let $p\in C_b$ be a non-singular point of $C_b$. We look at the Abel-Jacobi map \begin{gather*}
    \mathrm{AJ}: C_b\rightarrow \bar J_b, \\ c\mapsto \mathscr Hom(\mathcal I_c, \mathcal O_{C_b}(-p)).
\end{gather*}Note that $\mathrm{AJ}(C^{sm}_b)\subset J_b$ and $\mathrm{AJ}^*(\mathcal P_F)\cong F$ (from Lemma \ref{lem2.3}). The conclusion follows from Step 2 since $\mathcal P_F|_{J_b}$ is trivial.

\noindent{\bf Step 4.} We show that if $K_b = \{[F]\in \bar J_b: \mathcal P_F|_{J_b}\text{ is trivial} \}$, $\dim K_b\leq g-\tilde g_b$ where $\tilde g_b$ is the geometric genus of $C_b$.

Consider the action morphism $\mu_b : J_b \times \bar J_b \rightarrow \bar J_b$. For $F\in \bar J_b$, let $Z_b =\mu^{-1}(K_b) \cap J_b \times \{[F]\} \subset J_b \times \bar J_b$.  If $(L,F)\in Z_b$, then $L\otimes F|_{C^{sm}_b}\cong \mathcal O_{C^{sm}_b}$. Since $F$ is torsion free, $F$ is locally free on $C_b^{sm}$. Let $\tilde C_b\xrightarrow{\nu} C_b$ be the normalization. We have a surjective map $J_b \twoheadrightarrow \operatorname{Pic}^0(\tilde C_b)$ with kernel having dimension $g-\tilde g_b$. Since $F$ is fixed, the set of all such $L$ is fixed on $C^{sm}_b$. So $\nu^*L$ is given by a divisor supported on $\nu^{-1}(\operatorname{Sing}(C_b))$. The set of such divisors is countable and for each such divisor, the set of such line bundles $L$ is given by a translate of the kernel of $J_b \twoheadrightarrow \operatorname{Pic}^0(\tilde C_b)$. Thus the set of such $L$ lies in a countable union of subvarieties of $\dim g-\tilde g_b$ and in particular, does not contain any generic point of a subscheme of higher dimension. It follows that $\dim Z_b \leq g-\tilde g_b$. Since $\mu: J_b\times \bar J_b \rightarrow \bar J_b$ is smooth of relative dimension $g$, it is enough to show $\dim \mu^{-1}(K_b)\leq 2g-\tilde g_b$. Consider $\mu^{-1}(K_b)\hookrightarrow J_b\times \bar J_b \xrightarrow{\text{projection}}\bar J_b$. The fiber over $[F]\in \bar J_b$ is given by $\mu^{-1}(K_b)\cap J_b\times \{ [F]\}$. So every fiber has dimension $\leq g-\tilde g_b $. Therefore, we get that $\dim \mu^{-1}(K_b)\leq 2g-\tilde g_b$ which shows $\dim K_b\leq g-\tilde g_b$. 

\noindent {\bf  Step 5.} Let $S = \operatorname{Supp}FM_{\mathcal P}(\mathcal K)$ (or $\operatorname{Supp}FM_{\mathcal P^\vee}(\mathcal K)$). Then $S\cap {\bar J_b}\subset K_b\implies \dim S\cap \bar J_b\leq g-\tilde g_b$. From Lemma \ref{Severi Inequality} or Corollary \ref{cor2.2}, we know that $\dim B^{(\tilde g)}\leq g-\tilde g $. Let $\bar J^{(\tilde g)}=\pi^{-1}(B^{(\tilde g)})$.  Consider $S\cap \bar J^{(\tilde g)}\hookrightarrow \bar J^{(\tilde g)}\rightarrow B^{(\tilde g)}$. Each fiber has dimension $\leq g-\tilde g$. So $\dim S\cap \bar J^{(\tilde g)}\leq \dim B^{(\tilde g)}+g-\tilde g \leq \dim B$. Since $S$ is a finite union of the subschemes $S\cap \bar J^{(g)}$, we conclude that $\dim S\leq \dim B$. Thus we have shown $\operatorname{codim}_{\bar J}S\geq g$. 

This completes the proof.
\end{proof}
Given $\mathcal K_{i,1},\mathcal K_{i,2}$ $J$-equivariant locally free sheaves on $\bar J$, we can consider the object\begin{equation*}
    \mathcal M:= {p_{1,2,\dots,n}}_* \left (\bigotimes_{i=1}^n p_{i,n+1}^* (\mathcal P \otimes p_1^* \mathcal K_{i,1}\otimes p_2^* \mathcal K_{i,2}) \right ) \otimes \pi_{n}^*\omega_{\pi}[g]\in D^b(\bar J^n).
\end{equation*}
Here the bounded-ness of $\mathcal M$ is a consequence of flatness of $\mathcal P$ with respect to both the factors (see the discussion following \cite{MR3019453} Lemma $6.1$). 
\begin{prop}\label{prop2.6} [Arinkin, Maulik-Shen-Yin] We have the following codimension bound \begin{equation}
    \operatorname{codim}_{\bar J^n}\operatorname{Supp}\mathcal M \geq g.
\end{equation}
\end{prop}
\begin{proof}
    The proof is parallel to Proposition \ref{prop2.5}. We proceed in the same steps as before.

    \noindent{\bf Step 1.} Let $(F_1,F_2,\dots, F_n)\in \bar J^n$ be a closed point of $S:=\operatorname{Supp}M$. Set $b=\pi_n(F_1,F_2,\dots,F_n)$. By base change, we get that for some $i$,  \begin{equation}
        H^i\left (\bar J_b, \bigotimes_{i=1}^n \left (\mathcal P_{F_i}\otimes \mathcal O_{\bar J_b}^{\operatorname{rk}\mathcal K_{i,1}}\otimes {(\mathcal K_{i,2})}_b\right ) \right )\neq 0.
    \end{equation}
    \noindent{\bf Step 2.} As in Step 2 of Proposition  \ref{prop2.5}, we let $T_i$ be the $\mathbb G_m$-torsor corresponding to $\mathcal P_{F_i}|_{J_b}$. As before we see that the action of $J_b$ on $\bar J_b$ extends to an action of $T_i$ on $\mathcal P_{F_i}$ with $\mathbb G_m$ acting via dilation. Since $\mathcal O_{\bar J_b}$ and ${(\mathcal K_{i,2})}_b$ are $J_b$-equivariant, we let $T_i$ act on $\mathcal O_{\bar J_b}$ and ${(\mathcal K_{i,2})}_b$ via $T_i\rightarrow J_b$. If $T$ is the $\mathbb G_m$-torsor corresponding to $\left(\mathcal P_{F_1}\otimes \mathcal P_{F_2}\otimes \cdots \otimes \mathcal P_{F_n}\right)|_{J_b}$, we see that $T$ acts on $\bigotimes_{i=1}^n 
 \left( \mathcal P_{F_i}\otimes \mathcal O_{\bar J_b}^{\operatorname{rk}\mathcal K_{i,1}}\otimes {(\mathcal K_{i,2})}_b \right )$ and hence on $ H^i\left (\bar J_b, \bigotimes_{i=1}^n \left (\mathcal P_{F_i}\otimes \mathcal O_{\bar J_b}^{\operatorname{rk}\mathcal K_{i,1}}\otimes {(\mathcal K_{i,2})}_b\right ) \right )$ with $\mathbb G_m$ acting via dilation. Since $T$ is commutative we get a $1$-dimensional $T$ sub-module $V\subset H^i\left (\bar J_b, \bigotimes_{i=1}^n \left (\mathcal P_{F_i}\otimes \mathcal O_{\bar J_b}^{\operatorname{rk}\mathcal K_{i,1}}\otimes {(\mathcal K_{i,2})}_b\right ) \right ) $ which gives a splitting of $\mathbb G_m\rightarrow T$ It follows that $T$ is the trivial $\mathbb G_m$-torsor and hence $\left (\mathcal P_{F_1}\otimes \mathcal P_{F_2}\otimes \cdots \otimes \mathcal P_{F_n} \right )|_{J_b}$ is the trivial line bundle. 

 \noindent{\bf Step 3.} Using the Abel-Jacobi map $\mathrm{AJ}$, we see that this implies $\left (F_1\otimes F_2\otimes \cdots \otimes F_n \right )|_{C^{sm}_b}\cong \mathcal O_{C^{sm}_b}$ as in Proposition \ref{prop2.5}.

 \noindent{\bf Step 4.} Using analogous arguments as in Step 4 of Proposition  \ref{prop2.5}, we see that over a closed point $b\in B$, $$K_b:= \left\{(F_1, F_2,\dots, F_n)\in \underbrace{\bar J_b \times \bar J_b \times\cdots \times \bar J_b}_{n\text{ times}} : \left (F_1\otimes F_2\otimes \cdots \otimes F_n \right )|_{C^{sm}_b}\cong \mathcal O_{C^{sm}_b}  \right \}$$has dimension $\leq (n-1)g+\tilde g_b$ where $\tilde g_b$ is the geometric genus of $C_b$, whence $\operatorname{codim}_{\bar J^{n}}S\geq g$. 

\noindent{\bf Step 5.} We see that $S\cap \bar J_b^{n}$ has dimension $\leq g(n-1)+\tilde g_b$. In conjunction with Lemma \ref{Severi Inequality} or Corollary \ref{cor2.2}, we see that $\dim S\leq \dim B +(n-1)g$.

     This completes the proof.  

\end{proof}
Let $\iota: \bar J\times_B \bar J\hookrightarrow \bar J\times \bar J$ be the closed embedding. For $N\geq 0$ and $\mathcal K_{i,1}, \mathcal K_{i,2}$ $J$-equivariant locally free sheaves on $\bar J$, one can consider $$\bigotimes_{i=1}^n\iota_* \left (\mathcal P \otimes p_1^*\mathcal K_{i,1} \otimes p_2^* \mathcal K_{i,2} \right )\in D^b(\bar J\times \bar J). $$It is bounded and exact off $\bar J\times_B \bar J \subset \bar J\times \bar J$. Let $\mathcal K_1, \mathcal K_2$ be $J$-equivariant locally free sheaves on $\bar J$. Then one can consider $$\mathcal T=(\mathcal P^{-1}\otimes p_1^*\mathcal K_1\otimes p_2^*\mathcal K_2)\circ \bigotimes_{i=1}^N\iota_* \left (\mathcal P \otimes p_1^*\mathcal K_{i,1} \otimes p_2^* \mathcal K_{i,2} \right )\in D^b(\bar J\times \bar J).$$By definition (and by the flatness of $\mathcal P$ with respect to both factors), it is constructed by pulling back the two objects  $\bigotimes_{i=1}^n\iota_* \left (\mathcal P \otimes p_1^*\mathcal K_{i,1} \otimes p_2^* \mathcal K_{i,2} \right )$ ,  $\mathcal P^{-1}\otimes p_1^*\mathcal K_1\otimes p_2^*\mathcal K_2 $ to the triple product $\bar J \times\bar J\times_B \bar J$, taking their tensor product and then push-forwarding them onto the first and third factors.  
\begin{prop}\label{prop2.7}
   $\mathcal T$ is supported on a codimension $g$ subset of $\bar J\times_B \bar J$. 
\end{prop}
\begin{proof}
    The proof is very similar to Proposition \ref{prop2.5} and Proposition \ref{prop2.6}. We proceed in the following steps.

    \noindent{\bf Step 1.} We first show that the set-theoretic support of $\mathcal T$ is contained in $\bar J\times_B \bar J$. By definition and the projection formula, the object $\mathcal T$ is obtained via the push-forward of an object supported on \begin{gather*}
        \bar J\times_B \bar J\times_B \bar J \hookrightarrow \bar J\times \bar J\times_B \bar J.
    \end{gather*}Therefore it can be written as the push-forward through the following chain of maps \begin{gather*}
        \bar J\times_B \bar J\times_B \bar J \rightarrow \bar J\times_B \bar J \xrightarrow[]{\iota}\bar J\times \bar J
    \end{gather*}where the first morphism is the projection onto the first and third factors. This proves the claim.

    \noindent{\bf Step 2.}  Let $(F_1, F_2)\in \bar J_b \times \bar J_b$ be in $\operatorname{Supp}\mathcal T$. Then by base change, we see that
\begin{gather}\label{eq19}
    H^\bullet \left ( \bar J_b, \left( \bigotimes_{i=1}^Nj^*j_* \left (\mathcal P_{F_1}\otimes {(\mathcal K_{i,2})}_b\otimes \mathcal O_{\bar J_b}^{\operatorname{rk}\mathcal K_{i,1}}\right)  \right ) \otimes \mathcal P_{F_2}^\vee \otimes {\mathcal K_1}_b\otimes \mathcal O_{\bar J_b}^{\operatorname{rk}\mathcal K_2}  \right ) \neq 0
\end{gather}where $j: \bar J_b \hookrightarrow \bar J$ is the inclusion of the fiber over $b$. Since $\bar J$ is non-singular, $j^*j_*(\mathcal P_{F_1}\otimes {(\mathcal K_{i,2})}_b\otimes \mathcal O_{\bar J_b}^{\operatorname{rk}\mathcal K_{i,1}})$ is perfect on the possibly singular fiber $\bar J_b$. The next step is to describe the cohomology of this complex.  

\noindent{\bf Step 3.} We claim that each cohomology sheaf $\mathcal H^k(j^*j_*(\mathcal P_{F_1}\otimes {(\mathcal K_{i,2})}_b\otimes \mathcal O_{\bar J_b}^{\operatorname{rk}\mathcal K_{i,1}}))\in \operatorname{Coh}(\bar J_b)$ is a direct sum of finite copies of $\mathcal P_{F_1}\otimes {(\mathcal K_{i,2})}_b\otimes \mathcal O_{\bar J_b}^{\operatorname{rk}\mathcal K_{i,1}}$ on $\bar J_b$. 

To see this, note that we have the decomposition \begin{gather*}
    j^*j_*\mathcal O_{\bar J_b}\simeq \bigoplus_{k}\mathcal H^{k}(j^*j_*\mathcal O_{\bar J_b})[-k]
\end{gather*} where each term on the right is a free $\mathcal O_{\bar J_b}$-module. This follows from base change and the corresponding statement for the embedding $\{b\}\hookrightarrow B$.

Next we relate the two objects \begin{gather}\label{20}j^*j_*(\mathcal P_{F_1}\otimes {(\mathcal K_{i,2})}_b\otimes \mathcal O_{\bar J_b}^{\operatorname{rk}\mathcal K_{i,1}}) \ , \ j^*j_*\mathcal O_{\bar J_b}\otimes \mathcal P_{F_1}\otimes {(\mathcal K_{i,2})}_b\otimes \mathcal O_{\bar J_b}^{\operatorname{rk}\mathcal K_{i,1}} \in D^b(\bar J_b). \end{gather}We have by the projection formula 
\begin{equation*}\begin{aligned}
    & j_* \left (j^*j_*\mathcal O_{\bar J_b}\otimes \mathcal P_{F_1}\otimes {(\mathcal K_{i,2})}_b\otimes \mathcal O_{\bar J_b}^{\operatorname{rk}\mathcal K_{i,1}} \right ) \\  = & j_*\mathcal O_{\bar J_b} \otimes j_*\left (\mathcal P_{F_1}\otimes {(\mathcal K_{i,2})}_b\otimes \mathcal O_{\bar J_b}^{\operatorname{rk}\mathcal K_{i,1}}\right )  \\ =& j_*\left (\mathcal O_{\bar J_b}\otimes j^*j_*(\mathcal P_{F_1}\otimes {(\mathcal K_{i,2})}_b\otimes \mathcal O_{\bar J_b}^{\operatorname{rk}\mathcal K_{i,1}})\right)  \\ =& j_*j^*j_* \left ( \mathcal P_{F_1}\otimes {(\mathcal K_{i,2})}_b\otimes \mathcal O_{\bar J_b}^{\operatorname{rk}\mathcal K_{i,1}}\right).
    \end{aligned}
\end{equation*}
    Therefore the cohomology sheaves of the two objects in (\ref{20}) which are $\mathcal O_{\bar J_b}$-modules are isomorphic as $\mathcal O_{\bar J}$-modules. This forces them to be isomorphic as $\mathcal O_{\bar J_b}$-modules. 

    So we have \begin{equation*}
       \begin{aligned}
          & \mathcal H^k \left ( j^*j_*(\mathcal P_{F_1}\otimes {(\mathcal K_{i,2})}_b\otimes \mathcal O_{\bar J_b}^{\operatorname{rk}\mathcal K_{i,1}})\right ) \\ =&  \mathcal H^k\left (\ j^*j_*\mathcal O_{\bar J_b}\otimes \mathcal P_{F_1}\otimes {(\mathcal K_{i,2})}_b\otimes \mathcal O_{\bar J_b}^{\operatorname{rk}\mathcal K_{i,1}}\right) \\ =& \mathcal H^k(j^*j_*\mathcal O_{\bar J_b}) \otimes \mathcal P_{F_1}\otimes {(\mathcal K_{i,2})}_b\otimes \mathcal O_{\bar J_b}^{\operatorname{rk}\mathcal K_{i,1}}
       \end{aligned} 
    \end{equation*}where the last equality follows from flatness. The proof of the claim now follows from the decomposition of $j^*j_* \mathcal O_{\bar J_b}$. 

    \noindent{\bf Step 4.} By the claim of Step 3, we know that the object $j^*j_*(\mathcal P_{F_1}\otimes {\mathcal K_{i,2}}_b\otimes \mathcal O_{\bar J_b}^{\operatorname{rk}\mathcal K_{i,1}}) $ admits an increasing filtration induced by the standard truncation functors, whose graded pieces are direct sums of finite copies of $\mathcal P_{F_1}\otimes {\mathcal K_{i,2}}_b\otimes \mathcal O_{\bar J_b}^{\operatorname{rk}\mathcal K_{i,1}}$.  Hence the cohomology (\ref{eq19}) admits a filtration whose graded pieces are direct sums of finite copies of  \begin{equation}\label{eq21}
        H^\bullet \left ( \bar J_b,\left( \bigotimes_{i=1}^N \left (\mathcal P_{F_1}\otimes {\mathcal K_{i,2}}_b\otimes \mathcal O_{\bar J_b}^{\operatorname{rk}\mathcal K_{i,1}}\right)  \right ) \otimes \mathcal P_{F_2}^\vee \otimes {\mathcal K_1}_b\otimes \mathcal O_{\bar J_b}^{\operatorname{rk}\mathcal K_2}  \right ).
    \end{equation}
    In particular, we obtain from Equation (\ref{eq19}), there is a non-trivial cohomology of type (\ref{eq21}). 

    \noindent{\bf  Step 5.} By the analogous arguments as in Proposition \ref{prop2.6}, we show that this implies $(F_1|_{C^{sm}_b})^{\otimes N}\otimes (F_2|_{C^{sm}_b})\cong \mathcal O_{C^{sm}_b}$ which is a codimension $\tilde g_b$ condition on $\bar J_b^2$. Then, by Severi Inequality, Lemma \ref{Severi Inequality} or Corollary \ref{cor2.2}, we conclude that the support is contained in a codimension $g$ subset of $\bar J^2$ which completes the proof.  
    \end{proof}
    \begin{rmk}We note that results of Proposition  \ref{prop2.5}, Proposition \ref{prop2.6} and Proposition  \ref{prop2.7} still hold if $\mathcal L_1,\mathcal L_2, \mathcal K_{i,1}, \mathcal K_{i,2}, \mathcal K_1, \mathcal K_2, \mathcal K$, instead of being $J$-equivariant fit into exact triangles of $J$-equivariant locally free sheaves. So in particular, we can take $\mathcal L_\theta=\mathcal K_\theta=\mathcal K_{i,\theta}=p_\theta^*T_{\bar J}, \theta=1,2$ thanks to the short exact sequence \begin{equation}
         0\rightarrow \pi^*\Omega_B\rightarrow \Omega_{\bar J}\rightarrow \Omega_{\bar J/B}\rightarrow 0.
    \end{equation}We mention the following lemma, which also enables us to take wedge products of the co-tangent bundle in Propositions \ref{prop2.5},\ref{prop2.6} and \ref{prop2.7}. 
    \begin{lem}\label{lem2.9}
          Let $$0\rightarrow A_1\rightarrow A_2\rightarrow A_3\rightarrow 0$$ be a short exact sequence in $\operatorname{Coh}(\bar J)$. Then for the derived exterior power $\wedge^{k}A_2\in D^b(\bar J)$, we have finitely many objects in $D^b(\bar J)$, $B_0,B_1,B_2,\cdots, B_k=\wedge^k A_2$ together with morphisms $B_{j-1}\rightarrow B_j$ that fit into the exact triangles $$B_{j-1}\rightarrow B_j\rightarrow \wedge^{k-j}A_1\otimes^{\mathbb L}\wedge^j A_3 \xrightarrow{+1} B_{j-1}[+1].$$
    \end{lem}
    \begin{proof}
        For the proof of this statement, we refer to \cite{lichtenbaum2018constantfunctionalequationderived} Corollary 2.2. 
    \end{proof}
This will be used later to deduce similar codimension inequalities of the supports of appropriate $K$-theory classes twisted by Adams' operations on the co-tangent bundle.
    \end{rmk}

    \subsection{Adams Operations in $K$-theory} Having proved the Arinkin-type dimension bounds, we review the Adams operations in $K$-theory following \cite{MR910201}. This will be used to deduce the expected vanishing results. 
    
    For a closed immersion $i:X\hookrightarrow Y$ of finite type schemes, we define $K_X(Y)$ to be the Grothendieck group of bounded complexes of locally free sheaves exact off $X$. For an algebraic scheme $Y$, we denote by $K_\bullet (Y)$ its Grothendieck ring. In \cite{MR910201}, a sequence of Adams operations \begin{equation*}
        \psi^N: K_X(Y)\rightarrow K_X(Y)
    \end{equation*} were constructed using the $\lambda$-ring structure on $K_X(Y)$ given by the derived exterior product and the induction formula \begin{equation*}
        \psi^N-\psi^{N-1}\otimes \lambda^1 + \cdots + (-1)^{N-1}\psi^1\otimes \lambda^{N-1}+(-1)^NN\lambda^N =0.
    \end{equation*}
    The following formula was proved in \cite{MR1707198} Theorem $3.1$. For $F_\bullet\in K_X(Y)$, we have 
    \begin{equation}\label{eq23}
        {\operatorname{ch}^Y_X}_k(\psi^N(F_\bullet))=N^k{{\operatorname{ch}^Y_X}_k(F_\bullet)}
    \end{equation}where ${\operatorname{ch}^Y_X}_k\in A^k(X\rightarrow Y)$ denotes the degree $k$ part of the localized Chern character in the bi-variant Chow group $A^k(X\rightarrow Y)$.     

    We return to the closed immersion $\iota: \bar J\times_B \bar J\hookrightarrow \bar J\times\bar J$ and fix it. Since $\bar J\times \bar J$ is non-singular, capping\footnote{Here we use the cap product $\cap: K_X(Y)\times K_\bullet(Y)\rightarrow K_\bullet(X)$ as defined in \cite{MR910201} Section $1.7$} with $\mathcal O_{\bar J\times \bar J}\in K_\bullet(\bar J\times \bar J)$ induces an isomorphism \begin{equation}\label{eq24}
        \cap \mathcal O_{\bar J\times \bar J}: K_{\bar J\times_B \bar J}(\bar J\times \bar J) \xrightarrow[]{\simeq}K_\bullet (\bar J\times_B \bar J)
    \end{equation}whose inverse is the map that sends $F\in K_\bullet (\bar J\times_B \bar J)$ to any locally free resolution of $\iota_*F$ on $\bar J\times \bar J$ (\cite{MR910201} Lemma 1.9). Similarly, at the level of Chow groups, we have by \cite{MR1644323} Propsoition $17.4.2$, the isomorphism \begin{equation}\label{eq25}
        \cap [\bar J\times \bar J]: A^k(\bar J\times_B \bar J \xhookrightarrow{\iota}\bar J\times \bar J)\xrightarrow[]{\simeq} A_{2\dim \bar J-k}(\bar J\times_B \bar J).
    \end{equation}
    Under the isomorphisms \ref{eq24} and \ref{eq25}, the Adams operations with respect to the closed immersion $\iota: \bar J\times_B \bar J\hookrightarrow \bar J\times \bar J$ stand as \begin{equation}
        \psi^N : K_\bullet (\bar J\times_B \bar J) \rightarrow K_\bullet (\bar J\times_B \bar J)
    \end{equation}We denote the localized Chern character $\operatorname{ch}^{\bar J\times \bar J}_{\bar J\times_B \bar J}$ by $\widetilde{\operatorname{ch}}$ which under the isomorphisms $(\ref{eq24})$ and $(\ref{eq25})$ yields a morphism \begin{equation}
        \widetilde{\operatorname{ch}}: K_\bullet (\bar J\times_B \bar J)\rightarrow A_\bullet (\bar J\times_B \bar J).
    \end{equation}
Moreover Equation  (\ref{eq23}) stands as \begin{equation}
        \label{eq28} \widetilde{\operatorname{ch}}_k(\psi^N(F_\bullet)) =N^k\widetilde{\operatorname{ch}}_k(F_\bullet).
    \end{equation}
    We shall use this property to prove certain vanishing results in Section $4$. 
    
\section{Fourier Vanishing and Construction of Motivic Perverse Filtrations} 
In this section, we construct a family of good pair of correspondences in the sense of Definition  \ref{def1.2}. We shall use Fulton-McPharson's $\tau(-)$ class. We refer to \cite{MR1644323} Chapter 18 or \cite{MR4877374} Section 2.3 for a review of the properties of $\tau$. Fix a rational number $\alpha\in \mathbb Q$. Set \begin{equation}
    \mathfrak F_\alpha := \operatorname{td}(-T_{\bar J\times_B \bar J})^\alpha \cap \tau (\mathcal P), \ \mathfrak F^{-1}_{\alpha}:= \operatorname{td}(-T_{\bar J \times_B \bar J})^{1-\alpha}\operatorname{td}(-\pi_2^*T_B)\cap \tau(\mathcal P^{-1}).
\end{equation}
Our goal is to show that this pair of correspondences $(\mathfrak F_{\alpha},\mathfrak F^{-1}_\alpha)$ is a good pair.
First of all, since $\bar J\times_B \bar J$ is an lci scheme, we let $T_{\bar J\times_B \bar J}$ denote the virtual tangent bundle. More explicitly, we have \begin{equation*}
    T_{\bar J\times_B \bar J}= {T_{\bar J\times \bar J}|}_{\bar J\times_B \bar J}-N_{\bar J\times_B \bar J| \bar J\times \bar J} = p_1^*T_{\bar J}+p_2^*T_{\bar J} - \pi_2^*T_B\in K_\bullet(\bar J\times_B \bar J).
\end{equation*}
First of all we note that \begin{equation*}
        \widetilde{\operatorname{ch}}(-)=\operatorname{td}(-p_1^*T_{\bar J}-p_2^*T_{\bar J})\cap \tau(-).
         \end{equation*}
\begin{equation*}
    \begin{aligned}
        \mathfrak F_\alpha \circ \mathfrak F^{-1}_\alpha & = {p_{13}}_* \delta_{\bar J}^!\left (\mathfrak F^{-1}_\alpha\times \mathfrak F_\alpha \right ) \\ & = {p_{13}}_* \delta_{\bar J}^!\left (\operatorname{td}(-T_{\bar J \times_B \bar J})^{1-\alpha}\operatorname{td}(-\pi_2^*T_B)\cap \tau(\mathcal P^{-1})\times \operatorname{td}(-T_{\bar J\times_B \bar J})^\alpha \cap \tau (\mathcal P)  \right ) \\ & ={p_{13}}_* \left ( p_{12}^*\operatorname{td}(-T_{\bar J\times_B \bar J})^{1-\alpha}p_{23}^*\operatorname{td}(-T_{\bar J\times_B \bar J})^\alpha\operatorname{td}(-\pi_3^*T_B)\cap \delta_{\bar J}^!(\tau(\mathcal P^{-1}\boxtimes\mathcal P)) \right ) \\ & = {p_{13}}_* \left ( p_{12}^*\operatorname{td}(-T_{\bar J\times_B \bar J})^{1-\alpha}p_{23}^*\operatorname{td}(-T_{\bar J\times_B \bar J})^\alpha\operatorname{td}(-\pi_3^*T_B)\operatorname{td}(p_2^*T_{\bar J})\cap \tau(\delta_{\bar J}^*(\mathcal P^{-1}\boxtimes\mathcal P)) \right ) \\ & = p_1^* \operatorname{td}(-T_{\bar J})^\alpha p_2^*\operatorname{td}(-T_{\bar J})^{1-\alpha}  \cap \tau (\mathcal P\circ \mathcal P^{-1} )  \\ & =  p_1^* \operatorname{td}(-T_{\bar J})^\alpha p_2^*\operatorname{td}(-T_{\bar J})^{1-\alpha}   \cap {\Delta_{\bar J/B}}_* \operatorname{td}(T_{\bar J}) = [\Delta_{\bar J/B}].
    \end{aligned}
\end{equation*}Similarly, one has \begin{equation*}
    \mathfrak F^{-1}_\alpha \circ \mathfrak F_\alpha= [\Delta_{\bar J/B}].
\end{equation*}Let $B'\subset B $ be a non-empty open subvariety such that $\bar J_{B'}\rightarrow B'$ is an abelian scheme. It follows that there exists a non-empty open $U\subset B'$ such that the vector bundles ${T_B|}_U, {T_{\bar J}|}_{\pi^{-1}(U)}$ are trivial. In this case, $\mathcal P|_{\bar J_U \times_U \bar J_U} =\mathcal L$, the Poincare line bundle for the abelian scheme $\bar J_U \rightarrow U$ So we have \begin{equation*}{\mathfrak F_\alpha|}_U= \tau(\mathcal L) = \operatorname{ch}(\mathcal L)\cap [\bar J_U \times_U \bar J_U]\end{equation*}
where the second equality follows from the fact that $\bar J_U \times_U \bar J_U$ is non-singular. 

Thus we have established conditions $(1)$ and $(3)$ of Definition  \ref{def1.2}. The only thing left is Fourier Vanishing ( \ref{(FV)}), which we establish in the following section.
 \subsection{Fourier Vanishing}\label{sec3.1}
In this section we prove  (\ref{(FV)}) for the pair of algebraic cycles $\mathfrak F_\alpha, \mathfrak F^{-1}_\alpha$.  In fact we prove something slightly more general (see Corollary  \ref{cor3.1}) which we shall use later in Section \ref{sec3.3}. First, we note that if $q_1, q_2$ are the two projections $\bar J\times \bar J\rightarrow \bar J$, then \begin{equation}
    \mathcal T= (\mathcal P^{-1}\otimes p_1^*\mathcal K_1 \otimes p_2^*\mathcal K_2)\circ \left ((\iota_*\mathcal P)^{\otimes N}\otimes \bigotimes_{i=1}^{k_1}q_1^*\mathcal K_{i,1}^{\otimes N_{i,1}}\otimes \bigotimes_{i=1}^{k_2}q_2^*\mathcal K_{i,2}^{\otimes N_{i,2}}\textbf{} \right )
\end{equation}
is supported on a codimension $g$-subset of $\bar J\times_B \bar J$ by Proposition  \ref{prop2.7}. where $\mathcal K_1, \mathcal K_2, \mathcal K_{i,1}, \mathcal K_{i,2}$ are locally free sheaves that can be expressed in terms of an iterated extension via exact triangles of $J$-equivariant locally free sheaves (if two objects can be expressed as such, so is their derived tensor product). Let $S_d$ denote the symmetric group on $d$ symbols. Then we have a natural $ S_N$-action on $(\iota_*\mathcal P)^{\otimes N}$, $ S_{N_{i,1}}$-action on $ q_1^*\mathcal K_{i,1}^{\otimes N_{i,1}}$, $ S_{N_{i,2}}$-action on $ q_2^*\mathcal K_{i,2}^{\otimes N_{i,2}}$. 
Considering isotypic components, we see that
\begin{equation}
    (\mathcal P^{-1}\otimes p_1^*\mathcal K_1 \otimes p_2^*\mathcal K_2)\circ \left (\bigwedge^N(\iota_*\mathcal P)\otimes \bigotimes_{i=1}^{k_1}\bigwedge^{N_{i,1}}q_1^*\mathcal K_{i,1}\otimes \bigotimes_{i=1}^{k_2}\bigwedge^{N_{i,2}}q_2^*\mathcal K_{i,2}\textbf{} \right )
\end{equation}is supported on a codimension $g$ subset of $\bar J\times_B \bar J$. In particular, using Adams' operation, one has 
\begin{equation*}
    \left(\mathcal P^{-1}\otimes \bigotimes_{i=1}^{d_1}\psi^{n_{i,1}}(p_1^*\Omega_{\bar J})\otimes \bigotimes_{i=1}^{d_2} \psi^{n_{i,2}}(p_2^*\Omega_{\bar J}) \right) \circ \left (\psi^N(\mathcal P) \otimes \bigotimes_{i=1}^{k_1} \psi^{N_{i,1}}(p_1^*\Omega_{\bar J}) \otimes \bigotimes_{i=1}^{k_2} \psi^{N_{i,2}}(p_2^*\Omega_{\bar J}) \right )
\end{equation*} 
seen as an object of $K_\bullet (\bar J\times_B \bar J)$ is supported on a codimension $g$-subset since Adams operations are polynomials in the exterior powers. Now note that for any $\mathcal A, \mathcal B \in K_{\bullet}(\bar J\times_B \bar J)$, we have
\begin{equation}\label{eq32}\begin{aligned}
    \tau(\mathcal A) \circ \widetilde{\operatorname{ch}}(\mathcal B) & = {p_{13}}_* \delta_{\bar J}^! \left ( \widetilde{\operatorname{ch}}(\mathcal B)\times \tau(\mathcal A)\right ) \\ & =  {p_{13}}_* \delta_{\bar J}^! \left (\operatorname{td}(-p_1^*T_{\bar J}-p_2^*T_{\bar J})\cap \tau (\mathcal B)\times \tau (\mathcal A) \right) \\ & =\operatorname{td}(-p_1^*T_{\bar J})\cap  \tau\left ({p_{13}}_*(\delta_J^*(\mathcal A\times \mathcal B))\right) \\ &= \operatorname{td}(-p_1^*T_{\bar J})\cap \tau(\mathcal A\circ \mathcal B).
\end{aligned}
\end{equation}
In Equation (\ref{eq32}), taking, \begin{equation*}\begin{aligned}
\mathcal A= \mathcal P^{-1}\otimes \bigotimes_{i=1}^{d_1}\psi^{n_{i,1}}(p_1^*\Omega_{\bar J})\otimes \bigotimes_{i=1}^{d_2} \psi^{n_{i,2}}(p_2^*\Omega_{\bar J})
\\ \mathcal B=\psi^N(\mathcal P) \otimes \bigotimes_{i=1}^{k_1} \psi^{N_{i,1}}(p_1^*\Omega_{\bar J}) \otimes \bigotimes_{i=1}^{k_2} \psi^{N_{i,2}}(p_2^*\Omega_{\bar J}), \end{aligned} \end{equation*}we get the following vanishing since $\mathcal A\circ B$ is supported on a codimension $g$ subset of $\bar J\times_B \bar J$ by the discussion above:\begin{equation}\label{eq33}
    \sum_{i+j=D}\tau_i(\mathcal A)\circ \widetilde{\operatorname{ch}}_{j+\dim B}(\mathcal B)=0\text{ if }D<2g.
\end{equation}Plugging everything in Equation (\ref{eq33}), we get \begin{gather}\label{eq34}
    \begin{aligned}
      & \underbrace{\sum_{i_a}\sum_{j_b}\sum_{e_c}\sum_{f_d}}_{\sum _{a=1}^{d_1}i_a+ \sum _{b=1}^{d_2}j_b+k=D_1,\sum_{c=1}^{k_1} e_c +\sum_{d=1}^{k_2} f_{d}+l=D_2} \left(
        \prod_a n_{a,1}^{i_a}\operatorname{ch}_{i_a}(p_1^*\Omega_{\bar J})\prod_b n_{b,2}^{j_b}\operatorname{ch}_{j_b}(p_2^*T_{\bar J})\cap \tau_k(\mathcal P^{-1} ) \right )  \\ & \circ\left( N^{l+\dim B}\prod_c N_{c,1}^{e_c}\operatorname{ch}_{e_c}(p_1^*\Omega_{\bar J})\prod_d N_{d,2}^{f_d}\operatorname{ch}_{f_d}(p_2^*\Omega_{\bar J}) \cap \widetilde{ \operatorname{ch}}_{l+\dim B}(\mathcal{P}) \right )
       =  0 \text{ if } D_1+D_2< 2g.
    \end{aligned}
\end{gather}Since Equation (\ref{eq34}) is true for all $n_{\cdot,1}, n_{\cdot,2},N_{\cdot,1}, N_{\cdot,2}, N$'s we conclude
\begin{equation}
\begin{aligned}
   & \label{eq35}\prod_{a}\operatorname{ch}_{i_a}(p_1^*\Omega_{\bar J})\prod_{b}\operatorname{ch}_{j_b}(p_2^*\Omega_{\bar J})\cap \tau_k(\mathcal P^{-1})\circ \prod_{c}\operatorname{ch}_{e_c}(p_1^*\Omega_{\bar J})\prod_{d}\operatorname{ch}_{f_d}(p_2^*\Omega_{\bar J})\cap \widetilde{\operatorname{ch}}_{l+\dim B}(\mathcal P)=0 \\ &\text{if }\sum_a i_a +\sum_b j_b+\sum e_c+\sum_d f_d+k+l<2g. 
    \end{aligned}
\end{equation}

In fact, twisting $\mathcal A$ and $\mathcal B$ with $\psi^{(-)}(\pi_2^*T_B)$'s we can conclude the vanishing \begin{equation}
\begin{aligned}
   & \label{eq36}\prod_{r}\operatorname{ch}_{m_r}(\pi_2^*T_B)\prod_{a}\operatorname{ch}_{i_a}(p_1^*\Omega_{\bar J})\prod_{b}\operatorname{ch}_{j_b}(p_2^*\Omega_{\bar J})\cap \tau_k(\mathcal P^{-1})\\ &\circ \prod_{s}\operatorname{ch}_{n_s}(\pi_2^*T_B)\prod_{c}\operatorname{ch}_{e_c}(p_1^*\Omega_{\bar J})\prod_{d}\operatorname{ch}_{f_d}(p_2^*\Omega_{\bar J})\cap \widetilde{\operatorname{ch}}_{l+\dim B}(\mathcal P)=0 \\ &\text{if }\sum_r m_r+\sum_a i_a +\sum_b j_b+\sum_c e_c+\sum_d f_d+\sum_s n_s+k+l<2g. 
    \end{aligned}
\end{equation}
We get the following corollary from the above computation. 
\begin{cor}\label{cor3.1}We have the vanishing \begin{equation}\label{eq37}
        \left(\mathfrak F_{\beta}^{-1}\right)_i\circ \left(\mathfrak F_\alpha\right)_j =0 \text{ if }i+j<2g 
    \end{equation}for all $\alpha, \beta \in \mathbb Q$.
\end{cor}
\begin{proof}
    We see that \begin{equation*}
        \begin{aligned}
               & \left(\mathfrak F_{\beta}^{-1}\right)_i\circ \left(\mathfrak F_\alpha\right)_j \\ & = \left ( \operatorname{td}(-p_1^*T_{\bar J})^{1-\beta}\operatorname{td}(-p_2^*T_{\bar J})^{1-\beta}\operatorname{td}(-\pi_2^*T_B)^\beta\cap \tau(\mathcal P^{-1})\right )_{i} \\ &\circ \left (\operatorname{td}(p_1^*T_{\bar J})^{1-\alpha}\operatorname{td}(p_2^*T_{\bar J})^{1-\alpha}\operatorname{td}(\pi_2^*T_B)^{\alpha}\cap \widetilde{\operatorname{ch}}(\mathcal P)\right )_j \\ &= 0 \text{ if } i+j<2g
        \end{aligned}
    \end{equation*} because it is a sum of expressions of the form as in Equation (\ref{eq36}) once we expand out the powers of the Todd classes as polynomials in $\operatorname{ch}_{(-)}$.
\end{proof}In particular, this shows $\mathfrak F_\alpha, \mathfrak F^{-1}_\alpha$ satisfy (\ref{(FV)}). which is what we wanted to prove. 
\subsection{Construction of the Motivic Perverse Filtrations}\label{sec3.2}
Since $\mathfrak F_\alpha, \mathfrak F_{\alpha}^{-1}$ satisfy the conditions of Definition  \ref{def1.2}, we see that they form a good pair of correspondences and hence induce a motivic perverse filtration by Proposition  \ref{prop1.4}. We denote this motivic perverse filtration by $P^{(\alpha)}$. 

Given a motivic perverse filtration for $X\rightarrow B$, say  $P_k\mathfrak h(X)= (X, \mathfrak p_k,0)$, we get an induced filtration on $A_\bullet (X)$ given by \begin{equation}\label{eq38}
    P_k A_\bullet (X) ={\mathfrak p_k}_* A_\bullet (X).
\end{equation}
Note that if we set $\mathfrak q_{k+1}=\Delta_{X/B}-\mathfrak p_k$, then $\mathfrak q_{k+1}\circ \mathfrak p_k=0 $ and hence $$P_kA_\bullet (X)= \operatorname{Ker} \left ({\mathfrak q_{k+1}}_*: A_\bullet (X) \rightarrow A_\bullet (X) \right ).$$ 

Since we have constructed a family of motivic perverse filtrations for the compactified Jacobian fibration $\bar J\rightarrow B$, it is natural to ask if the perverse filtrations induced on $A_\bullet(\bar J)$ are the same\footnote{Note that we already know this at the level of cohomology thanks to the homlogical realization property in Definition \ref{def1.1}.}. 

The main result of this section is that the perverse filtrations induced at the level of Chow groups is the same for different $P^{(\alpha)}$'s. 

For the perverse filtration $P^{(\alpha)}$, we denote the corresponding projectors by $\mathfrak p^{(\alpha)}_k$ and $\mathfrak q^{(\alpha)}_{k+1}$. 
\begin{prop}\label{prop3.2}
For $\alpha, \beta \in \mathbb Q$, we have \begin{equation}\label{eq39}
    P^{(\alpha)}_kA_\bullet (\bar J)= P_{k}^{(\beta)}A_\bullet(\bar J).
\end{equation}    
\end{prop}
\begin{proof}
    It is enough to show that $P^{(\alpha)}_kA_\bullet (\bar J)\subset P^{(\beta)}_k A_\bullet (\bar J)$ for any $\alpha,\beta\in \mathbb Q$ which is equivalent to showing\begin{equation}
        \mathfrak q^{(\beta)}_{k+1}\circ \mathfrak{p}^{(\alpha)}_k=0. 
    \end{equation}
    To this end, note that we have \begin{equation*}
        \begin{aligned}
            &\mathfrak q^{(\beta)}_{k+1}\circ \mathfrak p_{k}^{(\alpha)}  \\ & = \left ( \sum_{i\geq k+1}{(\mathfrak F_\beta)}_i\circ (\mathfrak F^{-1}_{\beta})_{2g-i} \right ) \circ \left ( \sum_{j\leq k}{(\mathfrak F_\alpha)}_j\circ (\mathfrak F^{-1}_{\alpha})_{2g-j}\right ) \\ &= \sum_{i\geq k+1}\sum_{j\leq k} {(\mathfrak F_\beta)}_i\circ (\mathfrak F^{-1}_{\beta})_{2g-i} \circ {(\mathfrak F_\alpha)}_j\circ (\mathfrak F^{-1}_{\alpha})_{2g-j}\\ & = 0 \text{ by Corollary }\ref{cor3.1},
        \end{aligned}
    \end{equation*}which completes the proof.
\end{proof}
Proposition \ref{prop3.2} shows that each motivic perverse filtration $\left \{P^{(\alpha)}_k \mathfrak h(\bar J)\right\}_k$ after specializing to Chow, induces the same filtration at the level of Chow groups which we denote by $\left \{P_k A_\bullet (\bar J)\right \}_k$.

We end this section with a corollary that indicates how Fourier transform changes the perverse filtration at the level of Chow groups/cohomology. We use this result in Example \ref{ex4.8}.  
\begin{cor}\label{fourier-perverse-codimension}
    For every $\alpha\in \mathbb Q$, we have \begin{equation*}
        \mathfrak F_\alpha^{-1} \left( P_k A^d(\bar J)\right )\subset A^{\geq d+g-k}(\bar J)
    \end{equation*}  and similarly \begin{equation*}
         \mathfrak F_\alpha^{-1} \left( P_k H^d(\bar J; \mathbb Q)\right )\subset H^{\geq d+2g-2k}(\bar J;\mathbb Q).
    \end{equation*} 
\end{cor}
\begin{proof} The proof follows from an easy computation :
  \begin{equation*}\begin{split}
&\mathfrak F_{\alpha}^{-1}\circ \mathfrak p_k^{(\alpha)}= \sum_{i}\sum_{j\leq k}{(\mathfrak F_{\alpha}^{-1})}_i\circ {(\mathfrak F_\alpha)}_j \circ {(\mathfrak F_\alpha)}_{2g-j}^{-1}= {(\mathfrak F_\alpha )}_{\geq 2g-k}^{-1} \circ \mathfrak p^{(\alpha)}_k \\  \implies & \mathfrak F_\alpha^{-1}\left (P_kA^d(\bar J) \right ) \subset \mathfrak F^{-1}_{\geq 2g-k}\left ( A^d(\bar J)\right )\subset A^{\geq d+g-k}(\bar J).\end{split}
\end{equation*}The proof for cohomology is analogous. 
  
\end{proof}
\subsection{Multiplicativity of the Motivic Perverse Filtrations}\label{sec3.3}
In this section we raise the question of multiplicativity of the perverse filtration $P^{(\alpha)}$. We recall the notations introduced in Section \ref{sec1.1}. From Definition \ref{def1.1} we see that it is equivalent to the vanishing of the expression \begin{equation}\label{eq41}
    \mathfrak q^{(\alpha)}_{k+l+1}\circ [\Delta^{sm}_{\bar J/B}] \circ (\mathfrak p^{(\alpha)}_k\times \mathfrak p^{(\alpha)}_l)=0 \ \forall \ k,l.
\end{equation}
For notational convenience, we set $\mathfrak F=\mathfrak F_\alpha$ when $\alpha=1$ and similarly $\mathfrak F^{-1}$. It has been shown in \cite{MR4877374} (see Lemma $2.8$ and the discussion thereafter) that there exists \begin{equation} \label{eq42}\mathfrak C \in \operatorname{Corr}^{\geq -g}_B(\bar J, \bar J; \bar J)\end{equation} such that \begin{equation}\label{eq43}
 \mathfrak F \circ \mathfrak C \circ (\mathfrak F^{-1}\times \mathfrak F^{-1})=  [\Delta^{sm}_{\bar J/B}]\in \operatorname{Corr}^0_B (\bar J, \bar J; \bar J).
\end{equation}
\begin{prop}\label{prop3.3}
    The motivic perverse filtration $P^{(\alpha)}_k$ is multiplicative for every $\alpha$.
\end{prop}
\begin{proof}
    Since this is essentially a vanishing result, we begin by expanding \begin{equation}
        \begin{aligned}\label{eq44}
            &  \mathfrak q^{(\alpha)}_{k+l+1}\circ [\Delta^{sm}_{\bar J/B}] \circ (\mathfrak p^{(\alpha)}_k\times \mathfrak p^{(\alpha)}_l) \\ & =\mathfrak q^{(\alpha)}_{k+l+1} \circ \mathfrak F \circ \mathfrak C \circ (\mathfrak F^{-1}\circ \mathfrak p^{(\alpha)}_k \times \mathfrak F^{-1}\circ \mathfrak p_l^{(\alpha)})\\ & = \left ( \sum_{i_3\geq k+l+1}{(\mathfrak F_\alpha)}_{i_3}\circ {(\mathfrak F^{-1}_\alpha)}_{2g-i_3}\right )\circ \left (\sum_{j_3}\mathfrak F_{j_3} \right )    \circ \mathfrak C \\ &\circ\left (\left ( \sum_{j_1}\mathfrak F^{-1}_{j_1}\right ) \circ \left ( \sum_{i_1\leq k}{(\mathfrak F_\alpha)}_{i_1}\circ {(\mathfrak F^{-1}_\alpha)}_{2g-i_1} \right )\times \left (\sum_{j_2}\mathfrak F^{-1}_{j_2}\right )\circ \left ( \sum_{i_2\leq l}{(\mathfrak F_\alpha)}_{i_2}\circ {(\mathfrak F^{-1}_\alpha)}_{2g-i_2} \right )  \right ) \\
            & = \left ( \sum_{i_3\geq k+l+1}{(\mathfrak F_\alpha)}_{i_3}\circ {(\mathfrak F^{-1}_\alpha)}_{2g-i_3}\right )\circ \left (\sum_{j_3\geq k+l+1}\mathfrak F_{j_3} \right )    \circ \mathfrak C \\ &\circ\left (\left ( \sum_{j_1\geq 2g-k}\mathfrak F^{-1}_{j_1}\right ) \circ \left ( \sum_{i_1\leq k}{(\mathfrak F_\alpha)}_{i_1}\circ {(\mathfrak F^{-1}_\alpha)}_{2g-i_1} \right )\times \left (\sum_{j_2\geq 2g-l}\mathfrak F^{-1}_{j_2}\right )\circ \left ( \sum_{i_2\leq l}{(\mathfrak F_\alpha)}_{i_2}\circ {(\mathfrak F^{-1}_\alpha)}_{2g-i_2} \right )  \right ) \\ & \text{by Corollary }\ref{cor3.1} \\ 
            & = \mathfrak q^{(\alpha)}_{k+l+1}\circ \left (\sum_{j_3\geq k+l+1}\mathfrak F_{j_3} \right )    \circ \mathfrak C \circ\left (\left ( \sum_{j_1\geq 2g-k}\mathfrak F^{-1}_{j_1}\right ) \circ \mathfrak p_k^{(\alpha)}\times \left (\sum_{j_2\geq 2g-l}\mathfrak F^{-1}_{j_2}\right )\circ \mathfrak p_{l}^{(\alpha)}  \right ).
        \end{aligned}
    \end{equation} Note that \begin{equation}\label{eq45}
    \begin{aligned}
        \sum_{j_1\geq 2g-k}\mathfrak F^{-1}_{j_1}\in \operatorname{Corr}_B^{\geq g-k}(\bar J,\bar J), & \sum_{j_2\geq 2g-l}\mathfrak F^{-1}_{j_2}\in \operatorname{Corr}^{\geq g-l}_B (\bar J,\bar J), & \sum_{j_3\geq k+l+1}\mathfrak F_{j_3} \in \operatorname{Corr}_B^{\geq k+l+1-g}(\bar J,\bar J).
       \end{aligned}
    \end{equation}
From (\ref{eq42}) and \eqref{eq45}, we get that the sum in \eqref{eq44} $\in \operatorname{Corr}^{\geq g-k +g-l -g+k+l+1-g }_B(\bar J,\bar J;\bar J)=\operatorname{Corr}_B^{\geq 1}(\bar J,\bar J;\bar J)$. But since $\mathfrak q_{k+l+1}^{(\alpha)}, \mathfrak p^{(\alpha)}_k, \mathfrak p^{(\alpha)}_l$ are degree $0$ correspondences and $\Delta^{sm}_{\bar J/B}$ is a degree $0$ bi-correspondence, one has $$\mathfrak q^{(\alpha)}_{k+l+1}\circ [\Delta^{sm}_{\bar J/B}] \circ \left(\mathfrak p^{(\alpha)}_k\times \mathfrak p^{(\alpha)}_l\right)\in \operatorname{Corr}^0_B(\bar J, \bar J; \bar J),$$which then shows that  $$\mathfrak q^{(\alpha)}_{k+l+1}\circ [\Delta^{sm}_{\bar J/B}] \circ \left(\mathfrak p^{(\alpha)}_k\times \mathfrak p^{(\alpha)}_l\right)=0.$$This completes the proof.  
\end{proof}
\begin{rmk}
    Note that the proof of Proposition \ref{prop3.3}, with appropriate modifications, gives us the vanishing \begin{equation}
        \label{eq46} \mathfrak q_{k+l+1}^{(\alpha)}\circ [\Delta^{sm}_{\bar J/B}]\circ \left ( \mathfrak p_k^{(\beta)}\times \mathfrak p_l^{(\gamma)} \right )=0 \ \forall \ k,l, \ \forall \ \alpha,\beta, \gamma.
    \end{equation}
\end{rmk}
\section{Perversity of the Chern Classes}
In this section we prove our perversity bounds for the Chern classes $c_k(T_{\bar J})$. It is more natural to first prove this on a motivic level and then specialize to get the corresponding results in Chow groups and cohomology. 
\subsection{Some Vanishing Results} \label{sec4.1} In this section we prove some vanishing results akin to Section \ref{sec3.1}. Computing the $\tau$-classes of the Fourier-Mukai transforms, we see that since $\bar J$ is non-singular, 
\begin{equation}\label{eq47}
    \begin{aligned}
        \tau\left (FM_{\mathcal P\otimes p_1^* \mathcal L_1 \otimes p_1^*\mathcal L_2} (\mathcal K)\right) & = {{p_2}}_* \left (p_1^*\operatorname{ch}(\mathcal K)p_1^*\operatorname{ch}(\mathcal L_1)p_2^*\operatorname{ch}(\mathcal L_2)\cap \tau (\mathcal P) \right ), \\ \tau\left (FM_{\mathcal P^{-1}\otimes p_1^* \mathcal L_1 \otimes p_1^*\mathcal L_2}(\mathcal K) \right) & = {{p_2}}_* \left (p_1^*\operatorname{ch}(\mathcal K)p_1^*\operatorname{ch}(\mathcal L_1)p_2^*\operatorname{ch}(\mathcal L_2)\cap \tau (\mathcal P^{-1}) \right ).  
    \end{aligned}
\end{equation}
We re-write the RHS of Equation \eqref{eq47} as \begin{equation*} \begin{aligned} &{{p_2}}_* \left (p_1^*\operatorname{ch}(\mathcal K)p_1^*\operatorname{ch}(\mathcal L_1)p_2^*\operatorname{ch}(\mathcal L_2)\cap \tau (\mathcal P) \right )= \left( p_1^*\operatorname{ch}(\mathcal L_1)p_2^*\operatorname{ch}(\mathcal L_2)\cap \tau (\mathcal P)\right) (\operatorname{ch}(\mathcal K)), \\ &{{p_2}}_* \left (p_1^*\operatorname{ch}(\mathcal K)p_1^*\operatorname{ch}(\mathcal L_1)p_2^*\operatorname{ch}(\mathcal L_2)\cap \tau (\mathcal P^{-1}) \right )= \left ( p_1^*\operatorname{ch}(\mathcal L_1)p_2^*\operatorname{ch}(\mathcal L_2)\cap \tau (\mathcal P^{-1})\right ) (\operatorname{ch}(\mathcal K)) \end{aligned}\end{equation*}where we view $ p_1^*\operatorname{ch}(\mathcal L_1)p_2^*\operatorname{ch}(\mathcal L_2)\cap \tau (\mathcal P)$ and $p_1^*\operatorname{ch}(\mathcal L_1)p_2^*\operatorname{ch}(\mathcal L_2)\cap \tau (\mathcal P^{-1})$ as correspondences in $\operatorname{Corr}^\bullet_B(\bar J, \bar J)$. 
Note that from Proposition \ref{prop2.5}, we get \begin{equation}
\begin{aligned}
    \tau \left (FM_{\mathcal P \otimes p_1^* \mathcal L_1 \otimes p_2^*\mathcal L_2} \right ) (\mathcal K)\in A^{\geq g} (\bar J),  \\
    \tau \left (FM_{\mathcal P^{-1} \otimes p_1^* \mathcal L_1 \otimes p_2^*\mathcal L_2} \right ) (\mathcal K) \in A^{\geq g} (\bar J).
    \end{aligned}
\end{equation}Hence, from Equation \eqref{eq47} we get 
\begin{equation}\begin{aligned} \label{eq49}
 \sum_{i+m=g+k} \left ( p_1^*\operatorname{ch}(\mathcal L_1)p_2^*\operatorname{ch}(\mathcal L_2)\cap \tau (\mathcal P)  \right )_{i} \left (\operatorname{ch}_m (\mathcal K) \right ) =0, \\ \sum_{i+m=g+k} \left ( p_1^*\operatorname{ch}(\mathcal L_1)p_2^*\operatorname{ch}(\mathcal L_2)\cap \tau (\mathcal P^{-1})  \right )_{i} \left (\operatorname{ch}_m (\mathcal K) \right ) =0 \end{aligned}
\end{equation} if $k<g$. 

Using Lemma \ref{lem2.9} and the discussion in Section \ref{sec3.1}, we know that Proposition \ref{prop2.5} remains valid in the case of $K$-theory classes for $\mathcal L_1 = \bigotimes_{r}\psi^{s_{r}}(\pi^*T_B)\otimes \bigotimes_{a} \psi^{n_{a}}(\Omega_{\bar J})$, $\mathcal L_2= \bigotimes_{b}\psi^{m_{b}}(\Omega_{\bar J})$ and $\mathcal K= \psi^N(\Omega_{\bar J})$ since these $K$-theory classes are polynomials in the $J$-equivariant  bundle $\pi^*T_B$ and the exterior powers of $\Omega_{\bar J}$. Hence we get from Equation \eqref{eq49}, the following vanishing for $D<2g$ : 
\begin{equation}\label{eq50}
    \begin{aligned}
        & \underbrace{\sum_{i_r} \sum_{j_a}\sum_{k_b}}_{\sum i_r+\sum j_a + \sum k_b=D-m-l} \left ( \prod_{r}\operatorname{ch}_{i_r}(\psi^{s_{r}}(\pi^*T_B))\prod_{a}\operatorname{ch}_{j_a} (\psi^{n_a}(p_1^* \Omega_{\bar J}))\prod_{k}\operatorname{ch}_{k_b }(\psi^{m_b}(p_2^*\Omega_{\bar J}))\cap \tau _l(\mathcal P^{-1}) \right ) \\ & \left ( \operatorname{ch}_m(\psi^N(\Omega_{\bar J}) \right ) = 0.   \\ \implies & \underbrace{\sum_{i_r} \sum_{j_a}\sum_{k_b}}_{\sum i_r+\sum j_a + \sum k_b=D-m-l} \left ( \prod_{r}s_r^{i_r}\operatorname{ch}_{i_r}(\pi^*T_B)\prod_{a}n_a^{j_a}\operatorname{ch}_{j_a} (p_1^* \Omega_{\bar J})\prod_{b} m_b^{k_b}\operatorname{ch}_{k_b}(p_2^* \Omega_{\bar J}) \cap \tau_l(\mathcal P^{-1})\right ) \\ & \left ( N^m \operatorname{ch}_m(\Omega_{\bar J}) \right )= 0.
    \end{aligned}\end{equation}Since Equation \eqref{eq50} holds true for all $s_r,n_a, m_b, N$'s we get \begin{equation} \label{eq51}\begin{aligned}
         & \left (\prod_{r}\operatorname{ch}_{i_r}(\pi^*T_B)\prod_{a}\operatorname{ch}_{j_a}(p_1^*\Omega_{\bar J}) \prod_{b}\operatorname{ch}_{k_b}(p_2^*\Omega_{\bar J})\cap \tau_l (\mathcal P^{-1}) \right ) \left ( \operatorname{ch}_m (\Omega_{\bar J}) \right ) = 0 \\ &\text{if }\sum i_r +\sum j_a +\sum k_b +l+m < 2g. 
         \end{aligned}
    \end{equation}From Equation \eqref{eq51}, we get the following vanishing of the Fourier Mukai transforms of the Chern classes as follows: 
    \begin{equation}\label{eq52}
        {\left(\operatorname{\mathfrak F}_\alpha^{-1}\right )}_i \left(c_j(T_{\bar J})\right )=0 \text{ if }i+j<2g
    \end{equation}
    because $(\mathfrak F_{\alpha}^{-1})_i \left ( c_j(T_{\bar J})\right )$ is a sum of expressions on the LHS of (\ref{eq51}). 
    
    In fact, choosing $\mathcal K = \psi^{N_1}(\Omega_{\bar J}) \otimes \psi^{N_2}(\Omega_{\bar J}) \otimes \cdots \otimes \psi^{N_d}(\Omega_{\bar J})$, we can conclude the following vanishing of the Fourier-Mukai transforms for monomials in the Chern classes of $\bar J$, namely \begin{equation*}
        \left (\mathfrak F^{-1}_\alpha\right )_i \left (c_{j_1}(T_{\bar J})c_{j_2}(T_{\bar J})\cdots c_{j_d}(T_{\bar J}) \right )=0 \text{ if }i+j_1+j_2+\cdots +j_d <2g.
    \end{equation*}
    
    We shall use the vanishing \eqref{eq51} to give perversity bounds of the Chern classes. We end this section with the following corollary.
    \begin{cor}\label{cor4.1}
        We have the following identity for all $\alpha\in \mathbb Q$ \begin{equation}
            \mathfrak p^{(\alpha)}_k \left (c_k (T_{\bar J}) \right )=c_k (T_{\bar J}) 
        \end{equation}
    \end{cor}
 \begin{proof}
     The proof is a straight-forward computation; \begin{equation*}
         \begin{aligned}
              c_k(T_{\bar J}) & = \sum_{i}\left (\mathfrak F_{\alpha}\right )_{i} \circ \left ( \mathfrak F^{-1}_\alpha \right )_{2g-i} \left (c_k(T_{\bar J}) \right )  \\ & = \sum _{i\leq k} \left (\mathfrak F_{\alpha}\right )_{i} \circ \left ( \mathfrak F^{-1}_\alpha \right )_{2g-i} \left (c_k(T_{\bar J}) \right ) \\ & = \mathfrak p_k^{(\alpha)} \left (c_k(T_{\bar J}) \right )
         \end{aligned}
     \end{equation*}where the first equality follows from the fact $\mathfrak F_\alpha \circ \mathfrak F_\alpha^{-1}=[\Delta_{\bar J/B}]$ and the second equality follows from Equation \eqref{eq52}. This completes the proof.
 \end{proof}
\subsection{Proof of Theorem \ref{thm0.1}} \label{sec4.2}

We shall deduce Theorem \ref{thm0.1} as a motivic consequence. First of all, note that if $z\in A_k(\bar J)$, we consider ${\Delta_{\bar J/B}}_* (z)\in \operatorname{Corr}_B^k(\bar J, \bar J)$.  
\begin{thm}
If $c_k(T_{\bar J})$ is the $k$-th Chern class, we have for all $\alpha, \beta\in \mathbb Q$  \begin{equation}
    {\Delta_{\bar J/B}}_* \left ( c_k(T_{\bar J})\right )\circ \mathfrak p^{(\alpha)}_i\in \operatorname{Hom}_{\operatorname{CHM}(B)}\left ( P_i^{(\alpha)} \mathfrak h(\bar J), P^{(\beta)}_{i+k} \mathfrak h(\bar J)(k)\right ) \ \forall \ i
\end{equation}
\end{thm}
\begin{proof}
Note that if $\mathfrak p^{(\gamma)}_i(z)=z$ for some $z\in A_\bullet (\bar J)$, we have \begin{equation}\label{eq55}
\begin{aligned}
     \mathfrak q_{i+k+1}^{(\beta)}\circ {\Delta_{\bar J/B}}_* z \circ  \mathfrak p_i^{(\alpha)} & =  \mathfrak q_{i+k+1}^{(\beta)}\circ {\Delta_{\bar J/B}}_* \mathfrak p^{(\gamma)}_k(z) \circ  \mathfrak p^{(\alpha)}_i \\ & =  \mathfrak q^{(\beta)}_{i+k+1} \circ {p_{23}}_* \left ( p_1^* z \cap \left ( [\Delta^{sm}_{\bar J/B} ] \circ (\mathfrak p^{(\gamma)}_k \times \mathfrak p^{(\alpha)}_i ) \right ) \right ) \\ & = {p_{23}}_* \left ( p_1^* z \cap \left ( \mathfrak q^{(\beta)}_{i+k+1} \circ [\Delta^{sm}_{\bar J/B}\circ ] \circ (\mathfrak p_k^{(\gamma)} \times \mathfrak p_i^{(\alpha)})\right ) \right )\\ & = 0
\end{aligned}
\end{equation}where the last equality follows from Equation \eqref{eq46}. This completes the proof in view of Corollary \ref{cor4.1}.    
\end{proof}
\begin{thm}
    Under the identification $\operatorname{Corr}_B^d(B, \bar J)= A^d (\bar J)$, we have that \begin{equation}
        \label{eq56} c_k(T_{\bar J})\in A^k (P^{(\alpha)}_k \mathfrak h(\bar J)) \ \forall \ \alpha\in \mathbb Q.
    \end{equation}
\end{thm}
\begin{proof}
    By definition of the Chow groups of a motive (see Section \ref{sec1.1}), we have$$A^k\left (P^{(\alpha)}_k \mathfrak h(\bar J) \right ):= \operatorname{Hom}_{\operatorname{CHM}(B)}\left (\mathfrak h(B)(-k), P^{\alpha }_k\mathfrak h(\bar J)\right )= \mathfrak p_k^{(\alpha)} \circ\operatorname{Corr}^k_{B}(B, \bar J).$$Since $c_k(T_{\bar J}) = \mathfrak p^{(\alpha)}_k \left (c_k(T_{\bar J}) \right )= \mathfrak p_k^{(\alpha)} \circ c_k(T_{\bar J}) $ by Corollary \ref{cor4.1}, the conclusion follows.  
\end{proof}
Taking appropriate specializations, we get the following corollaries. \begin{cor}
   We have $c_k(T_{\bar J})\in P_k H^{2k}(\bar J; \mathbb Q)$. 
\end{cor}
\begin{cor}
   We have $c_k(T_{\bar J})\in P_k A^k(\bar J)$.
\end{cor}
\begin{cor}\label{cor4.6}
    The homological realization of ${\Delta_{\bar J/B}}_*c_k(T_{\bar J})$ under the Corti-Hanamura realization functor, i.e. taking cup-product with $c_k(T_{\bar J})$ induces a morphism $${}^p\tau_{\leq i+\dim B} \left (\pi_*\mathbb Q_{\bar J} \right ) \rightarrow \left ({}^p\tau _{\leq i+k +\dim B}\left ( \pi_* \mathbb Q_{\bar J} \right )\right)[2k]={}^p\tau_{\leq i-k+\dim B}(\pi_* \mathbb Q_{\bar J}[2k]).$$  
\end{cor}
 We briefly recall the notion of strong perversity as introduced in \cite{MR4792069}.
 \begin{defn}[\cite{MR4792069}, Section $1.1$] For a proper morphism $X\xrightarrow{f} B$ between complex non-singular varieties with equi-dimensional fibers $\gamma \in H^l(X, \mathbb Q)$ has strong perversity $c$ if the induced morphism $$\gamma: f_* \mathbb Q_X \rightarrow f_* \mathbb Q_X[l]$$after pushing forward along $f$ satisfies $$\gamma ({}^p\tau_{\leq i}(f_*\mathbb Q_X))\subset {}^p\tau_{\leq i+(c-l)} \left (f_* \mathbb Q_X[l]\right ).$$\end{defn}
So Corollary \ref{cor4.6} shows that $c_k(T_{\bar J})$ has strong perversity $k$ for every $k$. 

\subsection{Examples}\label{Examples}
We conclude this paper with the following two examples. The first example, {Example} \ref{ex4.7} shows that when there are no singular fibers, the perversity bound can be much smaller; on the other hand, we show in {Example} \ref{ex4.8} that, when there are singular fibers, our perversity bound is optimal. 
\begin{example}\label{ex4.7}
    Consider the case $\mathcal C\rightarrow B$ is smooth and proper so $\bar J= J$ and $J/B $ is an abelian scheme over $B$. We have $T_{\bar J}|_{\bar J_b}$ is trivial as $\bar J_b$ is an abelian variety and hence $T_{\bar J}=\pi^*L$ for some $L\in \operatorname{Pic}(B)$. It follows that $c_k(T_{\bar J})\in P_0 H^{2k}(\bar J;\mathbb Q) $ for every $k$ in this case. 
\end{example}
\begin{example}\label{ex4.8}
    Consider $S\xrightarrow{\pi}\mathbb P^1$, an elliptic $K3$ surface with integral fibers together with a section $s: \mathbb P^1 \rightarrow S$ so that $S$ is isomorphic to a compactified Jacobian family $\bar J$. We have $c_2(T_S)=24c_S$ where $c_S$ is the Beauville-Voisin class as shown in \cite{MR2047674}.   

    Set ${\mathfrak s}:= s_*([\mathbb P^1])\in A_1(S) $ and $\mathfrak f:= \pi^*([\operatorname{pt}])\in A_1(S)$. Note that $\mathfrak{sf}=c_S$. Consider $\Theta:=\mathfrak s+\mathfrak f \in A_1(S)$. Next, we set \begin{equation*}
        \bar{\mathfrak p}_0:= \pi_1^*\Theta, \ \bar{\mathfrak p}_2 :=\pi_2^*\Theta, \ \bar{\mathfrak p}_1:=[\Delta_{S/\mathbb P^1}]-\bar{\mathfrak p}_0-\bar{\mathfrak p}_2 
    \end{equation*}where $\pi_1,\pi_2: S\times_{\mathbb P^1}S\rightarrow S$ are the two projections. It can be checked by a direct computation that $\bar{\mathfrak p}_0, \bar{\mathfrak p}_1, \bar{\mathfrak p}_2\in \operatorname{Corr}^0_{\mathbb P^1}(S,S)$ are orthogonal projectors and we have the following motivic decomposition in $\operatorname{CHM}(\mathbb P^1)$\begin{equation*}
        \mathfrak h(S)=\mathfrak h_0(S)\oplus \mathfrak h_1(S)\oplus \mathfrak h_2(S), \ \mathfrak h_i(S)=(S,\bar{\mathfrak p}_i,0)
    \end{equation*}which specializes to a splitting of the perverse filtration on $R\pi_*\mathbb Q_S$. The following formulae have been established in \cite{bae2024generalizedbeauvilledecompositions}. If $\mathfrak F'$ is the Fourier transform induced by the cycle class $\operatorname{td}(-T_{S\times_{\mathbb P^1}S})^{\frac{1}{2}}\cap \tau (\mathcal P)$ with Fourier inverse $\mathfrak F'^{-1}$, then we get \begin{equation*} \begin{split}
        \mathfrak F'([S])=-\Theta+c_S, \ \mathfrak F'(c_S)=\mathfrak f, \ \mathfrak F'(\Theta)=[S]-\mathfrak f, \ \mathfrak F'(\mathfrak f)=-c_S, \\ \mathfrak F'^{-1}([S])=\Theta+c_S, \ \mathfrak F'^{-1}(c_S)= -\mathfrak f, \ \mathfrak F'^{-1}(\Theta)=-[S]-\mathfrak f, \ \mathfrak F'^{-1}(\mathfrak f)=c_S.
    \end{split}\end{equation*} 
 It follows that $\mathfrak F'^{-1}(c_2(T_S))\in A^1(S)$.  By Corollary \ref{fourier-perverse-codimension}, we have $\mathfrak F'^{-1} \left ( P_1A^2(S)\right )\subset A^{\geq 1+ 2 -1}(S)=A^{\geq 2}(S)$ which shows that $c_2(T_S)\in P_2A^{2}(S)\backslash P_1A^2(S)$.   
 \end{example}
 \bibliographystyle{alpha}
 \bibliography{citations}

@article {MR429908,
    AUTHOR = {Altman, Allen B. and Kleiman, Steven L.},
     TITLE = {Compactifying the {J}acobian},
   JOURNAL = {Bull. Amer. Math. Soc.},
  FJOURNAL = {Bulletin of the American Mathematical Society},
    VOLUME = {82},
      YEAR = {1976},
    NUMBER = {6},
     PAGES = {947--949},
      ISSN = {0002-9904},
   MRCLASS = {14H40 (14D05)},
  MRNUMBER = {429908},
MRREVIEWER = {Knud\ L\o nsted},
       DOI = {10.1090/S0002-9904-1976-14229-2},
       URL = {https://doi.org/10.1090/S0002-9904-1976-14229-2},
}

@incollection {MR498546,
    AUTHOR = {Altman, Allen B. and Iarrobino, Anthony and Kleiman, Steven
              L.},
     TITLE = {Irreducibility of the compactified {J}acobian},
 BOOKTITLE = {Real and complex singularities ({P}roc. {N}inth {N}ordic
              {S}ummer {S}chool/{NAVF} {S}ympos. {M}ath., {O}slo, 1976)},
     PAGES = {1--12},
 PUBLISHER = {Sijthoff \& Noordhoff, Alphen aan den Rijn},
      YEAR = {1977},
      ISBN = {90-286-0097-3},
   MRCLASS = {14C05 (14H40)},
  MRNUMBER = {498546},
MRREVIEWER = {T.\ Oda},
}

@article {MR1707198,
    AUTHOR = {Kurano, Kazuhiko and Roberts, Paul C.},
     TITLE = {Adams operations, localized {C}hern characters, and the
              positivity of {D}utta multiplicity in characteristic {$0$}},
   JOURNAL = {Trans. Amer. Math. Soc.},
  FJOURNAL = {Transactions of the American Mathematical Society},
    VOLUME = {352},
      YEAR = {2000},
    NUMBER = {7},
     PAGES = {3103--3116},
      ISSN = {0002-9947,1088-6850},
   MRCLASS = {13A35 (13D15 14C17 14C35)},
  MRNUMBER = {1707198},
MRREVIEWER = {Jaime-Luis\ Garcia-Roig},
       DOI = {10.1090/S0002-9947-00-02589-7},
       URL = {https://doi.org/10.1090/S0002-9947-00-02589-7},
}

@article {MR2047674,
    AUTHOR = {Beauville, Arnaud and Voisin, Claire},
     TITLE = {On the {C}how ring of a {$K3$} surface},
   JOURNAL = {J. Algebraic Geom.},
  FJOURNAL = {Journal of Algebraic Geometry},
    VOLUME = {13},
      YEAR = {2004},
    NUMBER = {3},
     PAGES = {417--426},
      ISSN = {1056-3911,1534-7486},
   MRCLASS = {14C15 (14C25 14J28)},
  MRNUMBER = {2047674},
MRREVIEWER = {James\ D.\ Lewis},
       DOI = {10.1090/S1056-3911-04-00341-8},
       URL = {https://doi.org/10.1090/S1056-3911-04-00341-8},
}

@article {MR910201,
    AUTHOR = {Gillet, H. and Soul\'e, C.},
     TITLE = {Intersection theory using {A}dams operations},
   JOURNAL = {Invent. Math.},
  FJOURNAL = {Inventiones Mathematicae},
    VOLUME = {90},
      YEAR = {1987},
    NUMBER = {2},
     PAGES = {243--277},
      ISSN = {0020-9910,1432-1297},
   MRCLASS = {14C17 (13D05 14B15 14C40 19E08)},
  MRNUMBER = {910201},
MRREVIEWER = {G.\ Horrocks},
       DOI = {10.1007/BF01388705},
       URL = {https://doi.org/10.1007/BF01388705},
}

@misc{lichtenbaum2018constantfunctionalequationderived,
      title={The constant in the functional equation and derived extrior powers}, 
      author={Stephen Lichtenbaum},
      year={2018},
      eprint={1810.08644},
      archivePrefix={arXiv},
      primaryClass={math.AG},
      url={https://arxiv.org/abs/1810.08644}, 
}

@article {MR4877374,
    AUTHOR = {Maulik, Davesh and Shen, Junliang and Yin, Qizheng},
     TITLE = {Perverse {F}iltrations and {F}ourier {T}ransforms},
   JOURNAL = {Acta Math.},
  FJOURNAL = {Acta Mathematica},
    VOLUME = {234},
      YEAR = {2025},
    NUMBER = {1},
     PAGES = {1--69},
      ISSN = {0001-5962,1871-2509},
   MRCLASS = {14F08 (14H40)},
  MRNUMBER = {4877374},
       DOI = {10.4310/acta.2025.v234.n1.a1},
       URL = {https://doi.org/10.4310/acta.2025.v234.n1.a1},
}

@article {MR3259038,
    AUTHOR = {Maulik, Davesh and Yun, Zhiwei},
     TITLE = {Macdonald formula for curves with planar singularities},
   JOURNAL = {J. Reine Angew. Math.},
  FJOURNAL = {Journal f\"ur die Reine und Angewandte Mathematik. [Crelle's
              Journal]},
    VOLUME = {694},
      YEAR = {2014},
     PAGES = {27--48},
      ISSN = {0075-4102,1435-5345},
   MRCLASS = {14C05 (14H20 14J40)},
  MRNUMBER = {3259038},
MRREVIEWER = {Roberto\ Notari},
       DOI = {10.1515/crelle-2012-0093},
       URL = {https://doi.org/10.1515/crelle-2012-0093},
}

@article {MR2653248,
    AUTHOR = {Ngo, Bao Chau},
     TITLE = {Le lemme fondamental pour les alg\`ebres de {L}ie},
   JOURNAL = {Publ. Math. Inst. Hautes \'Etudes Sci.},
  FJOURNAL = {Publications Math\'ematiques. Institut de Hautes \'Etudes
              Scientifiques},
    NUMBER = {111},
      YEAR = {2010},
     PAGES = {1--169},
      ISSN = {0073-8301,1618-1913},
   MRCLASS = {22E35 (11S37 14D23 14G35 22E50)},
  MRNUMBER = {2653248},
MRREVIEWER = {R.\ P.\ Langlands},
       DOI = {10.1007/s10240-010-0026-7},
       URL = {https://doi.org/10.1007/s10240-010-0026-7},
}

@article {MR2525735,
    AUTHOR = {de Cataldo, Mark Andrea A. and Migliorini, Luca},
     TITLE = {The decomposition theorem, perverse sheaves and the topology
              of algebraic maps},
   JOURNAL = {Bull. Amer. Math. Soc. (N.S.)},
  FJOURNAL = {American Mathematical Society. Bulletin. New Series},
    VOLUME = {46},
      YEAR = {2009},
    NUMBER = {4},
     PAGES = {535--633},
      ISSN = {0273-0979,1088-9485},
   MRCLASS = {14C30 (14F05 14F43 18E30)},
  MRNUMBER = {2525735},
MRREVIEWER = {Adrian\ Langer},
       DOI = {10.1090/S0273-0979-09-01260-9},
       URL = {https://doi.org/10.1090/S0273-0979-09-01260-9},
}

@incollection {MR3752459,
    AUTHOR = {de Cataldo, Mark Andrea},
     TITLE = {Perverse sheaves and the topology of algebraic varieties},
 BOOKTITLE = {Geometry of moduli spaces and representation theory},
    SERIES = {IAS/Park City Math. Ser.},
    VOLUME = {24},
     PAGES = {1--58},
 PUBLISHER = {Amer. Math. Soc., Providence, RI},
      YEAR = {2017},
      ISBN = {978-1-4704-3574-5},
   MRCLASS = {14F05},
  MRNUMBER = {3752459},
MRREVIEWER = {Fei\ Ye},
}

@misc{hausel2025pwmathcalh2,
      title={{$P=W$} via $\mathcal{H}_2$}, 
      author={Tamas Hausel and Anton Mellit and Alexandre Minets and Olivier Schiffmann},
      year={2025},
      eprint={2209.05429},
      archivePrefix={arXiv},
      primaryClass={math.AG},
      url={https://arxiv.org/abs/2209.05429},
 note={Arxiv:2209.05429}
}

@article {MR4792069,
    AUTHOR = {Maulik, Davesh and Shen, Junliang},
     TITLE = {The {$P=W$} conjecture for {${\rm GL}_n$}},
   JOURNAL = {Ann. of Math. (2)},
  FJOURNAL = {Annals of Mathematics. Second Series},
    VOLUME = {200},
      YEAR = {2024},
    NUMBER = {2},
     PAGES = {529--556},
      ISSN = {0003-486X,1939-8980},
   MRCLASS = {14H60 (14D20 14F45)},
  MRNUMBER = {4792069},
       DOI = {10.4007/annals.2024.200.2.3},
       URL = {https://doi.org/10.4007/annals.2024.200.2.3},
}

@misc{bae2024generalizedbeauvilledecompositions,
      title={On {G}eneralized {B}eauville {D}ecompositions}, 
      author={Younghan Bae and Davesh Maulik and Junliang Shen and Qizheng Yin},
      year={2024},
      eprint={2402.08861},
      archivePrefix={arXiv},
      primaryClass={math.AG},
      note={Arxiv:2402.08861}
}

@article {MR3019453,
    AUTHOR = {Arinkin, Dima},
     TITLE = {Autoduality of {C}ompactified {J}acobians for {C}urves with {P}lane {S}ingularities},
   JOURNAL = {J. Algebraic Geom.},
  FJOURNAL = {Journal of Algebraic Geometry},
    VOLUME = {22},
      YEAR = {2013},
    NUMBER = {2},
     PAGES = {363--388},
      ISSN = {1056-3911,1534-7486},
   MRCLASS = {14H40 (14D20 14F05 14H20 14K30)},
  MRNUMBER = {3019453},
MRREVIEWER = {H.\ Lange},
       DOI = {10.1090/S1056-3911-2012-00596-7},
       URL = {https://doi.org/10.1090/S1056-3911-2012-00596-7},
}

@article {MR2915476,
    AUTHOR = {Arinkin, D.},
     TITLE = {Cohomology of line bundles on compactified {J}acobians},
   JOURNAL = {Math. Res. Lett.},
  FJOURNAL = {Mathematical Research Letters},
    VOLUME = {18},
      YEAR = {2011},
    NUMBER = {6},
     PAGES = {1215--1226},
      ISSN = {1073-2780,1945-001X},
   MRCLASS = {14H40 (14F05)},
  MRNUMBER = {2915476},
       DOI = {10.4310/MRL.2011.v18.n6.a11},
       URL = {https://doi.org/10.4310/MRL.2011.v18.n6.a11},
}

@article {MR1763656,
    AUTHOR = {Corti, Alessio and Hanamura, Masaki},
     TITLE = {Motivic decomposition and intersection {C}how groups. {I}},
   JOURNAL = {Duke Math. J.},
  FJOURNAL = {Duke Mathematical Journal},
    VOLUME = {103},
      YEAR = {2000},
    NUMBER = {3},
     PAGES = {459--522},
      ISSN = {0012-7094,1547-7398},
   MRCLASS = {14F42 (14C15 19E20 55N33)},
  MRNUMBER = {1763656},
MRREVIEWER = {Luca\ Barbieri Viale},
       DOI = {10.1215/S0012-7094-00-10334-1},
       URL = {https://doi.org/10.1215/S0012-7094-00-10334-1},
}

@article {MR2330158,
    AUTHOR = {Corti, Alessio and Hanamura, Masaki},
     TITLE = {Motivic decomposition and intersection {C}how groups. {II}},
   JOURNAL = {Pure Appl. Math. Q.},
  FJOURNAL = {Pure and Applied Mathematics Quarterly},
    VOLUME = {3},
      YEAR = {2007},
    NUMBER = {1},
     PAGES = {181--203},
      ISSN = {1558-8599,1558-8602},
   MRCLASS = {14C15 (14F43)},
  MRNUMBER = {2330158},
MRREVIEWER = {Jean-Louis\ Cathelineau},
       DOI = {10.4310/PAMQ.2007.v3.n1.a6},
       URL = {https://doi.org/10.4310/PAMQ.2007.v3.n1.a6},
}

@misc{maulik2024algebraiccycleshitchinsystems,
      title={Algebraic cycles and Hitchin systems}, 
      author={Davesh Maulik and Junliang Shen and Qizheng Yin},
      year={2024},
      eprint={2407.05177},
      archivePrefix={arXiv},
      primaryClass={math.AG},
      url={https://arxiv.org/abs/2407.05177},
note={Arxiv:2407.05177}
}

@book {MR1644323,
    AUTHOR = {Fulton, William},
     TITLE = {Intersection theory},
    SERIES = {Ergebnisse der Mathematik und ihrer Grenzgebiete. 3. Folge. A
              Series of Modern Surveys in Mathematics [Results in
              Mathematics and Related Areas. 3rd Series. A Series of Modern
              Surveys in Mathematics]},
    VOLUME = {2},
   EDITION = {Second},
 PUBLISHER = {Springer-Verlag, Berlin},
      YEAR = {1998},
     PAGES = {xiv+470},
      ISBN = {3-540-62046-X; 0-387-98549-2},
   MRCLASS = {14C17 (14-02)},
  MRNUMBER = {1644323},
       DOI = {10.1007/978-1-4612-1700-8},
       URL = {https://doi.org/10.1007/978-1-4612-1700-8},
}

@article {MR1133323,
    AUTHOR = {Deninger, Christopher and Murre, Jacob},
     TITLE = {Motivic {D}ecomposition of {A}belian {S}chemes and the {F}ourier
              transform},
   JOURNAL = {J. Reine Angew. Math.},
  FJOURNAL = {Journal f\"ur die Reine und Angewandte Mathematik. [Crelle's
              Journal]},
    VOLUME = {422},
      YEAR = {1991},
     PAGES = {201--219},
      ISSN = {0075-4102,1435-5345},
   MRCLASS = {14K05 (11G09 11G10 14C15)},
  MRNUMBER = {1133323},
MRREVIEWER = {Takeshi\ Ooe},
}

@article {MR4870047,
    AUTHOR = {Beilinson, Alexander and Bernstein, Joseph and Deligne, Pierre
              and Gabber, Ofer},
     TITLE = {Faisceaux pervers},
   JOURNAL = {Ast\'erisque},
  FJOURNAL = {Ast\'erisque},
      YEAR = {2018},
    NUMBER = {100},
     PAGES = {vi+180},
      ISSN = {0303-1179,2492-5926},
      ISBN = {978-2-85629-878-7},
   MRCLASS = {32S30 (14F08 18G80 19F27 32S60 55N33)},
  MRNUMBER = {4870047},
}

@article {MR961600,
    AUTHOR = {Diaz, Steven and Harris, Joe},
     TITLE = {Ideals associated to deformations of singular plane curves},
   JOURNAL = {Trans. Amer. Math. Soc.},
  FJOURNAL = {Transactions of the American Mathematical Society},
    VOLUME = {309},
      YEAR = {1988},
    NUMBER = {2},
     PAGES = {433--468},
      ISSN = {0002-9947,1088-6850},
   MRCLASS = {14B07 (14H20)},
  MRNUMBER = {961600},
MRREVIEWER = {Philippe\ Carbonne},
       DOI = {10.2307/2000919},
       URL = {https://doi.org/10.2307/2000919},
}

@article {MR4602420,
    AUTHOR = {Maulik, Davesh and Shen, Junliang},
     TITLE = {Cohomological {$\chi$}-independence for moduli of
              one-dimensional sheaves and moduli of {H}iggs bundles},
   JOURNAL = {Geom. Topol.},
  FJOURNAL = {Geometry \& Topology},
    VOLUME = {27},
      YEAR = {2023},
    NUMBER = {4},
     PAGES = {1539--1586},
      ISSN = {1465-3060,1364-0380},
   MRCLASS = {14H60 (14J60 14N35 32S60)},
  MRNUMBER = {4602420},
MRREVIEWER = {Hongjie\ Yu},
       DOI = {10.2140/gt.2023.27.1539},
       URL = {https://doi.org/10.2140/gt.2023.27.1539},
}

@article{maulik2023fourier,
  title={Fourier-{M}ukai transforms and the decomposition theorem for integrable systems},
  author={Maulik, Davesh and Shen, Junliang and Yin, Qizheng},
  journal={arXiv preprint arXiv:2301.05825},
  year={2023}
}

@article {MR4015227,
    AUTHOR = {Melo, Margarida and Rapagnetta, Antonio and Viviani, Filippo},
     TITLE = {Fourier-{M}ukai and autoduality for compactified {J}acobians.
              {I}},
   JOURNAL = {J. Reine Angew. Math.},
  FJOURNAL = {Journal f\"ur die Reine und Angewandte Mathematik. [Crelle's
              Journal]},
    VOLUME = {755},
      YEAR = {2019},
     PAGES = {1--65},
      ISSN = {0075-4102,1435-5345},
   MRCLASS = {14F08 (14B07 14D20 14H40 14J32)},
  MRNUMBER = {4015227},
MRREVIEWER = {Chunyi\ Li},
       DOI = {10.1515/crelle-2017-0009},
       URL = {https://doi.org/10.1515/crelle-2017-0009},
}
\end{document}